\documentclass[11pt]{amsart}
\usepackage[utf8]{inputenc}
\usepackage{lmodern}
\usepackage{amsmath, amsthm, amssymb, amsfonts,bm}
\usepackage[normalem]{ulem}
\usepackage{hyperref}

\usepackage{verbatim} 
\usepackage{longtable}

\usepackage{mathtools}

\usepackage{tikz}
\usetikzlibrary{decorations.pathmorphing}
\tikzset{snake it/.style={decorate, decoration=snake}}
\usepackage{tikz-cd}
\usetikzlibrary{positioning}

\usepackage{caption}

\usepackage{tikz-cd}
\usetikzlibrary{arrows}

\theoremstyle{plain}
\newtheorem{thm}{Theorem}[section]
\newtheorem{cor}[thm]{Corollary}
\newtheorem{lem}[thm]{Lemma}
\newtheorem{prop}[thm]{Proposition}

\theoremstyle{definition}

\newtheorem{example}[thm]{Example}

\theoremstyle{remark}
\newtheorem{rmk}[thm]{Remark}

\DeclareFontFamily{OT1}{rsfs}{}
\DeclareFontShape{OT1}{rsfs}{n}{it}{<-> rsfs10}{}
\DeclareMathAlphabet{\curly}{OT1}{rsfs}{n}{it}

\newcommand\Spec{\operatorname{Spec}}

\newcommand{\Coh}{\mathrm{Coh}}
\newcommand{\Pic}{\mathop{\rm Pic}\nolimits}

\def\Z{\mathbb{Z}}

\def\C{\mathbb{C}}

\def\P{\mathbb{P}}

\def\O{\mathcal{O}}

\def\a{\alpha}
\def\b{\beta}

\def\m{\mu}
\def\t{\theta}
\def\s{\sigma}
\def\sm{\Sigma}

\def\ts{\widetilde{\Sigma}}

\def\m{\mathcal}
\def\f{\mathfrak}
\def\c{\mathscr}

\def\bb{\mathbb}

\def\td{\widetilde}
\def\wh{\widehat}

\def\Spec{\textrm{Spec}}
\def\deg{\textrm{deg}}
\def\ch{\textrm{ch}}
\def\c{\textrm{c}}
\def\rk{\textrm{rk}}
\def\dim{\textrm{dim}}
\def\rad{\textrm{rad}}
\def\ker{\textrm{ker}}

\def\Prym{\textrm{Prym}}
\def\Nm{\textrm{Nm}}
\def\D{\textrm{D}}
\def\Pr{\textrm{Pr}}

\usepackage[backend=biber, sorting = nyt , giveninits=true,doi=false,url=false, maxbibnames=99]{biblatex} 
\usepackage{vmargin}
\setpapersize{A4}
\setmarginsrb{27mm}{12mm}{27mm}{12mm}%
{12mm}{10mm}{5mm}{10mm}

\usepackage{tikz}
\usepackage{lmodern}
\usetikzlibrary{decorations.pathmorphing}

\begin{document}
\title{Moduli spaces of modules over even Clifford algebras and Prym varieties}

\author{Jia Choon Lee}
\address{Peking University, Beijing International Center for Mathematical Research, Jingchunyuan Courtyard \#78, 5 Yiheyuan Road, Haidian District,
Beijing 100871, China}
\email{jiachoonlee@gmail.com}

\begin{abstract}
A conic fibration has an associated sheaf of even Clifford algebras on the base. In this paper, we study the relation between the moduli spaces of modules over the sheaf of even Clifford algebras and the Prym variety associated to the conic fibration. In particular, we construct a rational map from the moduli space of modules over the sheaf of even Clifford algebras to the special subvarieties in the Prym variety, and check that the rational map is birational in some cases. As an application, we get an explicit correspondence between instanton bundles of minimal charge on cubic threefolds and twisted Higgs bundles on curves. 
\end{abstract}

\baselineskip=14.5pt
\maketitle

\setcounter{tocdepth}{1} 

\tableofcontents
\section{Introduction}
It is well-known that a smooth cubic threefold is irrational since the famous work of Clemens and Griffiths \cite{clemens1972intermediate}. They observed that if a threefold is rational, then its intermediate Jacobian must
be isomorphic to a product of Jacobians of curves. The problem is then reduced to comparing the intermediate Jacobians of cubic threefolds with the Jacobians of curves as principally polarized abelian varieties by studying the singularity loci of their theta divisors. 

Despite the success in dimension three, the rationality problem for cubic fourfolds is still not well understood. A categorical approach to the problem is to study the bounded derived category $D^b(Y_4)$ of a given smooth cubic fourfold $Y_4$ and it always admits a semiorthogonal decomposition: 
\begin{equation}\label{semiorthogonal decomposition: cubic}
    \D^b(Y_4) \cong \langle \m Ku(Y_4), \O_{Y_4}, \O_{Y_4}(1), \O_{Y_4}(2)\rangle 
\end{equation}
where $\m Ku(Y_4) := \langle\O_{Y_4}, \O_{Y_4}(1), \O_{Y_4}(2)\rangle ^\perp $ is now known as the Kuznetsov component. It is conjectured by Kuznetsov \cite{kuznetsov2010cubic} that a smooth cubic fourfold is rational if and only if the Kuznetsov component $\m Ku(Y_4)$ is equivalent to the category of a K3 surface. While the conjecture has been checked to hold in some cases, the general conjecture remains unsolved.

Since $\m Ku(Y_4)$ is expected to capture the geometry of $Y_4$, an attempt to extract information out of the triangulated category $\m Ku(Y_4)$ is to construct Bridgeland stability conditions on $\m Ku(Y_4)$ and consider their moduli spaces of stable objects \cite{bayer2017stability}\cite{bayer2021stability}. In dimension three, one can define in the same way $\m Ku(Y_3):= \langle \O_{Y_3},\O_{Y_3}(1)\rangle^\perp$ for a smooth cubic threefold $Y_3$ and it can be shown that $\m Ku(Y_3)$ reconstructs the Fano surface of lines of $Y_3$ as a moduli space of stable objects with suitable stability conditions. The reconstruction of the Fano surface of lines then determines the intermediate Jacobian $J(Y_3)$ \cite{bernardara12cubic}. Alternatively, it is observed that instanton bundles of minimal charge on $Y_3$ are objects in $\m Ku(Y_3)$. Then by the work of Markushevich-Tikhomirov and others \cite{markushevich1998}\cite{Iliev2000}\cite{Druel}\cite{beauville02cubic}, it is shown that the moduli space of instanton bundles of minimal charge on $Y_3$ is birational to the intermediate Jacobian $J(Y_3)$. So the Kuznetsov components $\m Ku(Y_3)$ and $\m Ku(Y_4)$ can be thought of as the categorical counterparts of the intermediate Jacobian whose success in the rationality problem of cubic threefolds fits well into the philosophy of Kuznetsov's conjecture in $n=4.$  

On the other hand, we know that both $Y_3$ and $Y_4$ admit a conic bundle structure. More precisely, for $n=3,4$, let $l_0\subset Y_n$ be a line that is not contained in a plane in $Y_n.$ Then the blow-up $\td{Y}_n:=Bl_{l_0}(Y_n)\subset Bl_{l_0}(\P^{n+1})$ of $Y_n$ along $l_0$ projects to a projective space $\P^{n-1}$, denoted by $p:\td{Y}_n\to \P^{n-1}$. The map $p$ is a conic bundle whose discriminant locus $\Delta_n$ is a degree 5 hypersurface. The idea of realizing a cubic hypersuface birationally as a conic bundle can be used to study its rationality. Following the idea of Mumford, it is shown that the intermediate Jacobian $J(Y_3)\cong J(\td{Y}_3)$ is isomorphic to $\Prym (\td{\Delta}_3,\Delta_3)$ where $\td{\Delta}_3$ is the double cover parametrizing the irreducible components of the degenerate conics over $\Delta_3$. By analyzing the difference between Prym varieties and the Jacobian of curves as principally polarized abelian varieties, it is again shown that a smooth cubic threefold is irrational. 

The conic bundle structure of a cubic hypersurface also provides us information at the level of derived category. A quadratic form on a vector space defines the Clifford algebra which decomposes into the even and odd parts. We can apply the construction of Clifford algebra relatively for the conic bundles $\td{Y}_n$ which is viewed as a family of conics over $\P^{n-1}$, and obtain a sheaf of even Clifford algebras $\m B_0$ on $\P^{n-1}$. The bounded derived category $\D^b(\P^{n-1},\m B_0)$ of $\m B_0$-modules appears as a component of the semiorthogonal decomposition of the conic bundle $\td{Y}_n$ \cite{kuznetsov2008derived}: 
\begin{equation}\label{semiorthogonal decomposition: quadric fibration}
    \D^b( \td{Y}_n)  = \langle \D^b(\P^{n-1}, \m B_0), p^*\D^b(\P^{n-1})\rangle.
\end{equation}
In the case $n=3,4$, by comparing the semiorthogonal decompositions in (\ref{semiorthogonal decomposition: cubic}), (\ref{semiorthogonal decomposition: quadric fibration}) and the one for blowing up, one can show that there are embedding functors
\begin{equation}\label{embedding functor}
    \Xi _ n : \m Ku(Y_n) \hookrightarrow \D^b(\P^{n-1},\m B_0)
\end{equation}
for $n=3,4$ (see \cite{bernardara12cubic}, \cite{bayer2021stability}). The functors $\Xi_n$ are useful in the study of $\m Ku(Y_n)$. For example, when $n=3,4$, the construction of Bridgeland stability conditions on $\m Ku(Y_n)$ carried out in \cite{bernardara12cubic}\cite{bayer2021stability} uses the embedding functors $\Xi_n$ as one of the key steps. Also, in the work of Lahoz-Macr\`{i}-Stellari \cite{Lahoz15ACM}, the functor $\Xi_3$ is used to provide a birational map between the moduli space of instanton bundles of minimal charge and the moduli space of $\m B_0$-modules. The moduli space of $\m B_0$-modules was also first considered in \cite{Lahoz15ACM}. 

Motivated by the relations found in the case of cubic threefolds as described above: \linebreak Prym/intermediate Jacobian, intermediate Jacobian/moduli space of instanton bundles, and instanton bundles/$\m B_0$-modules, it is natural to search for a relation between $\m B_0$-modules and the Prym varieties. In this paper, we will focus on three dimensional conic fibrations over $\P^2$ (not necessarily obtained from a cubic threefold) and study the relation between the moduli spaces of $\m B_0$-modules and the Prym varieties. 

Let $p:X\to \P^2$ be a three dimensional conic fibration over $\P^2$ i.e. a flat quadric fibration of relative dimension 1 over $\P^2$ with simple degeneration i.e. the degenerate conics cannot be double lines and the discriminant locus $\Delta$ in $\P^2$ is smooth. Let $\pi:\td{\Delta}\to\Delta$ be the double cover parametrizing the irreducible components of degenerate conics over $\Delta$. We consider the moduli space $\f M_{d,e}$ of semistable $\m B_0$-modules whose underlying $\O_{\P^2}$-modules have fixed Chern character $(0,2d,e)$. This means that the $\m B_0$-modules are supported on plane curves. The moduli space $\f M_{d,e}$ comes with a morphism $\Upsilon: \f M_{d,e} \to |\O_{\P^2}(d)|$ defined by sending a $\m B_0$-module to its schematic (Fitting) support.   

On the Prym side, by the work of Welters \cite{Welters81AJ_Isogenies} and Beauville \cite{beauville1982sous}, each linear system on $\Delta$ defines a so-called special subvariety in the Prym variety $\Prym(\td{\Delta},\Delta)$ of the \'{e}tale double cover $\pi:\td{\Delta}\to \Delta$. We apply the construction to the linear system $|L_d|=\lvert \O_{\P^2}(d)|_\Delta\rvert$. For each $k$, there is an induced morphism $\pi^{(k)}:\td{\Delta}^{(k)}\to \Delta^{(k)}$ between the $k$-th symmetric products of $\td{\Delta}$ and $\Delta$. As the linear system $|L_d|$ can be considered as a subvariety in $\Delta^{(k)}$ for $k= \deg(\Delta)\cdot d$, we define the variety of divisors lying over $|L_d|$ as $W_d=(\pi^{(k)})^{-1} (|L_d|)\subset \td{\Delta}^{(k)}$. Then the image of $W_d$ under the Abel-Jacobi map $\td{\a}:\td{\Delta}^{(k)}\to J^k\td{\Delta}$ lies in the two components of (the translate of) Prym varieties, which are called the special subvarieties. The variety $W_d$ consists of two components $W_d=W_d^0\cup W_d^1$, each of which maps to $|L_d|$. For $d<\deg(\Delta)$, we have $|\O_{\P^2}(d)|\cong |L_d|$, and we denote by $U_d\subset |\O_{\P^2}(d)|$ the open subset of smooth degree $d$ curves intersecting $\Delta$ transversally. 

The main construction in this paper is to construct a morphism over $U_d$
\begin{equation}
    \Phi: \f M_{d,e}|_{U_d} \to W_d|_{U_d}
\end{equation}
for $d=1,2, i= 0,1$ and $ e\in \Z$. Here $\f M_{d,e}|_{U_d}\subset \f M_{d,e}$ is the open subset of $\m B_0$-modules supported on curves in $U_d$ and $W_d|_{U_d}\subset W_d$ is the open subset consisting of the degree $k$ divisors in $\td{\Delta}$ whose images in $\Delta$ represent degree $k$  reduced divisors in $L_d$. Moreover, the morphism $\Phi$ is used to prove the following (see Theorem \ref{parity and birational type }):
\begin{thm}
   Let $d=1, 2$ and $e\in \Z$. 
\begin{enumerate}
    \item The moduli space $\f M_{d,e}$ is birational to one of the two connected components $W_d^i$ of $W_d.$
    \item If $\f M_{d,e}$ is birational to $W_d^i$, then $\f M_{d,e+1}$ is birational to $W_d^{1-i}.$ In particular, the birational type of $\f M_{d,e}$  only depends on $d$ and $(e$ mod $2)$.
\end{enumerate}
\end{thm}

The idea behind the construction of the morphism $\Phi$ is based on a study of the representations of degenerate even Clifford algebras. A $\m B_0$-module in $\f M_{d,e}|_{U_d}$ is supported on a plane curve $C$ which intersects the discriminant curve $\Delta$ in finitely many points. Then the $\m B_0$-module restricted to each of these points $p\in C\cap \Delta$ gives rise to a representation of a degenerate even Clifford algebra, which is in turn shown to be equivalent to a representation of the path algebra associated the quiver 
\begin{equation*}
\begin{tikzcd}
        +\arrow[r,bend left,"\a"]&-\arrow[l,bend left,"\b"]
\end{tikzcd}
\end{equation*}
with relations $\a\b=\b\a=0.$ Then an analysis of the representations coming from $\m B_0$-modules reveals that there are natural candidates determining canonically the required lift of the point $p\in C\cap \Delta$ to $\td{\Delta}$, hence an element in $W_d|_{U_d}$.

By composing with the Abel-Jacobi map $\td{\a}:\td{\Delta}^{(k)}\to J^k\td{\Delta}$ that maps to the varieties $J^k\td{\Delta}$ of degree $k$ invertible sheaves on $\td{\sm}$, we obtain a rational map (choose a base point in $J^k\td{\Delta}$ to identify with $J^0\td{\Delta}$) 
\begin{equation}
   \td{\a}\circ \Phi: \f M_{d,e}\dashrightarrow \Prym(\td{\Delta},\Delta)    
\end{equation}
whose image is an open subset of a special subvariety. 

Next, we apply the result above to the case of cubic threefolds. In \cite{kuznetsov2012instanton} and \cite{Lahoz15ACM}, it is observed that instanton bundles of minimal charge are objects in $\m Ku(Y_3)$. The authors use the functor $\Xi_3: \m Ku(Y_3)\hookrightarrow \D^b(\P^2,\m B_0)$ to deduce a birational map between the moduli space $\f M_{Y_3}$ of instanton bundles of minimal charge on $Y_3$ and the moduli space $\f M_{2,-4}$ of $\m B_0$-modules. In this case, the rational map $\f M_{2,-4}\dashrightarrow \Prym(\td{\Delta},\Delta)$  actually turns out to be birational. Hence, by composing the birational maps, we get 
\begin{equation}\label{instanton moduli= prym}
    \f M_{Y_3}\dashrightarrow \Prym(\td{\Delta},\Delta)
\end{equation}
As a point in $\Prym(\td{\Delta},\Delta)$ can be interpreted as a twisted Higgs bundle on $\Delta$ by the spectral correspondence \cite{beauville1989spectral}, the birational map (\ref{instanton moduli= prym}) gives an explicit correspondence between instanton bundles of minimal charge on $Y_3$ and twisted Higgs bundles on $\Delta$. 

Moreover, as mentioned above, the moduli space of instanton bundles of minimal charge is birational to the intermediate Jacobian $J(Y_3)$, so the birational map (\ref{instanton moduli= prym}) gives a modular interpretation of the classical isomorphism $J(Y_3)\cong \Prym(\td{\Delta},\Delta)$ in terms of instanton bundles of minimal charge and twisted Higgs bundles via $\m B_0$-modules. From this viewpoint, the embedding functor $\Xi_3: \m Ku(Y_3)\hookrightarrow \D^b(\P^2,\m B_0)$ provides a reinterpretation of the  classical isomorphism $J(Y_3)\cong \Prym(\td{\Delta},\Delta)$. 

Philosophically, the result allows us to think of $D^b(\P^2,\m B_0)$ as the categorical counterpart of $\Prym(\td{\Delta},\Delta)$ associated to a conic bundle, just as $\m Ku(Y_3)$ is the categorical counterpart of $J(Y_3).$

\subsection{Convention}
Throughout this paper we work over the complex numbers $\C$. All modules in this paper are assumed to be left modules. For a morphism $f: X\to Y$ of two spaces (schemes or stacks) and a subspace $Z\subset Y$, we will denote by $X|_Z: = X\times_Y Z$ the fiber product and $f|_Z: X|_Z\to Z$.

\subsection{Acknowledgement}
This paper is written as part of my Ph.D. thesis at the University of Pennsylvania. I would like to thank both my advisors Ron Donagi and Tony Pantev for their constant help, many discussions and encouragement. I would also like to thank Emanuele Macr\`{i} for many useful discussions and Alexander Kuznetsov for all the useful and detailed comments on the earlier drafts of this paper. I would also like to thank the reviewers for their careful reading and insightful comments. 
\section{Special subvarieties in Prym varieties}
In this section, we recall the special subvariety construction of Prym varieties following the work of Welters \cite{Welters81AJ_Isogenies} and Beauville \cite{beauville1982sous}. Let $\pi: \ts\to \sm$ be an \'{e}tale double cover of two smooth curves. Then we denote by $\Nm:J\ts\to J\sm$ the norm map on the Jacobians of curves. We also have the induced map $\Nm^d: J^d\ts\to J^d\sm $ where $J^d\td{\sm}$ and $J^d\sm$ are the varieties of degree $d$ invertible sheaves on $\td{\sm}$ and $\sm$ respectively. We recall that the Prym variety is defined to be $\Prym(\td{\sm},\sm):= \ker(\Nm)^\circ$ i.e. the connected component of the kernel of the norm map. . 

Suppose $g^r_d$ is a linear system of degree $d$ and (projective) dimension $r$ on $\sm$. Consider the Abel-Jacobi maps: 
\begin{align*}
        \td{\a}: \ts^{(d)}\to J^d\ts ,& \quad \td{x}_1 + ... +\td{x}_d \mapsto \O(\td{x}_1 + ...+ \td{x}_d )\\
        \a: \sm^{(d)}\to J^d\sm,& \quad x_1+...+x_d\mapsto \O(x_1+...+x_d)
\end{align*} 
they fit in the following commutative diagram
\begin{equation}
    \begin{tikzcd}
    \ts^{(d)}\arrow[r,"\td{\a}"]\arrow[d,"\pi^{(d)}"]&J^d\ts\arrow[d,"\Nm^d"]\\
    \sm^{(d)}\arrow[r,"\a"]& J^d\sm
    \end{tikzcd}
\end{equation}
The linear system $g^r_d\cong \P^r$ is naturally a subvariety of $\sm^{(d)}$. We assume that the linear system $g^r_d$ contains a reduced divisor, so that $g^r_d$ is not contained in the branch locus of $\pi^{(d)}$.

Now, we define $W = (\pi^{(d)})^{-1}(g^r_d)$ as the preimage of $g^r_d$ in $\ts^{(d)}$. The image of $W$ under $\td{\a}$ is denoted by $V= \td{\a}(W)$. The inverse image $(\Nm^d)^{-1}(\a(g^r_d))$ consists of two disjoint components, each of which is isomorphic to the Prym variety $\Prym(\td{\sm},\sm)$ by translation, and we denote them by $\Pr^0$ and $\Pr^1$. By construction, we have that $V\subset (\Nm^d)^{-1}(\a(g^r_d))$, so $V$ also has two disjoint components $V^i\subset \Pr^i$ for $i=0,1$. Hence, $W$ also breaks into a disjoint union of two subvarieties $W^0$ and $W^1$ such that $\td{\a} (W^i) = V^i$. It is proved in \cite[Proposition 8.8]{Welters81AJ_Isogenies} that $W^0$ and $W^1$ are the connected components of $W.$ Welters \cite{Welters81AJ_Isogenies} called $W$ the variety of divisors on $\td{\sm}$ lying over the $g^r_d$ and the two connected components $W^0$ and $W^1$ the \emph{halves} of the variety of divisors $W$. The subvarieties $V^i$ are called the \emph{special subvarieties} of $\Pr^i$ associated to the linear system $g^r_d.$ 

\begin{rmk}\label{involution by even card subsets}
Let $\sigma:\td{\sm}\to \td{\sm}$ be the involution on $\td{\sm}.$ By \cite[Lemma 1]{mumford71theta}, a line bundle $L\in\ker (\Nm)$ can always be written as $L\cong M\otimes \sigma^* (M^\vee)$ such that if $\deg(M)\equiv 0$ (resp. $1)$ mod $2$, then $L\in \Pr^0$ (resp. $\Pr^1$). It follows that if $L\in \Pr^i$, then $L\otimes \O(x -\sigma (x))\in \Pr^{1-i}$ where $x\in \sm$. This implies that if $x_1+... + x_d\in W^i$, then the divisor $\sigma(x_1)+... + x_d = (\sigma(x_1) - x_1) + (x_1+...+x_d)  $ is contained in $W^{1-i}$. In particular, we see that if we involute an even number of points $x_k$ in $x_1+...+x_d\in W^i$, then the resulting divisor lies in the same component, i.e. $\sum_{k\in I} \sigma(x_k)+\sum_{j\not \in I}x_j\in W^{1-i} $ if $x_1+...+x_d\in W^{i}$ and $I$ has even cardinality. 
\end{rmk}

Let $\overline{\sm}:= \Z/2\Z\times \sm$ be the constant group scheme over $\sm$. The trivial double cover $p:\overline{\sm}\to \sm$ also induces a morphism on its $d$-th symmetric products $\overline{\sm}^{(d)}\to \sm^{(d)}$. Let $U\subset \sm^{(d)}$ be the open subset of reduced effective divisors. 
\begin{prop}\label{group scheme}
The scheme $G' :=\overline{\sm}^{(d)}|_U$ is a group scheme over $U$. 
\end{prop}
\begin{proof}
Note that the map $G'\to U$ is \'{e}tale. The multiplication map $m:\overline{\sm}\times _\sm \overline{\sm} \to \overline{\sm}$ induces the map $m^{(d)}:(\overline{\sm}\times _\sm \overline{\sm})^{(d)} \to \overline{\sm}^{(d)}$. On the other hand, the natural projections $pr_j:\overline{\sm}\times_\sm \overline{\sm}\to \overline{\sm}$ induces the maps $pr_j^{(d)}:(\overline{\sm}\times_\sm \overline{\sm})^{(d)}\to \overline{\sm}^{(d)}$ and so $r:(p^{(d)} \circ pr_{1}^{(d)})^{-1}(U) \to G'\times_{U} G'$ by universal property. It is easy to see that $r$ is bijective on closed points. As $G'$ and $U$ are smooth, so $G'\times_{U} G'$ is also smooth and hence normal. Therefore, $r$ is an isomorphism. Then we define the multiplication map on $G'$ to be 
\begin{equation*}
    G'\times_UG' \xrightarrow{r^{-1}}(p^{(d)} \circ pr_{1}^{(d)})^{-1}(U) \xrightarrow{m^{(d)}|_U} G'
\end{equation*}
At closed points, the group multiplication $G'(\C)\times_{U(\C)}G'(\C)\to G'(\C) $ simply sends
\begin{equation*}
    \left(\sum_{i=1}^d(\a_i,x_i), \sum_{i=1}^d(\b_i,x_i)\right) \to \sum_{i=1}^d(\a_i+\b_i,x_i)
\end{equation*}
over $\sum_{i=1}^d x_i\in U(\C)$, where $\a_i,\b_i\in \Z/2\Z$.

The trivial double cover $\overline{\sm}\to \sm$ always has a section $\sm\to \overline{\sm}$ mapping to $q^{-1}(0)$ where $q:\overline{\sm}\to \Z/2\Z$ is the projection map. The identity map is defined as the restriction of $\sm^{(d)}\to\overline{\sm}^{(d)} $ to $U$, i.e. $e:U\to G'$. 

The inverse map is simply the identity map $\iota: G'\to G'$. 
\end{proof}

The projection map $q:\overline{\sm}\to \Z/2\Z$ induces $G'\to \overline{\sm}^{(d)}\xrightarrow{q^{(d)}} (\Z/2\Z)^{(d)}$ and there is the summation map $\Z/2\Z^{(d)}\to \Z/2\Z$, and we denote by $s: G' \to \Z/2\Z $ the composition of the two maps. Then we define the preimage $G:= s^{-1}(0)$. 

\begin{cor}
The scheme $G$ is a group scheme over $U$.
\end{cor}
We can denote a closed point of $G$ as $\sum (\lambda_i, x_i) $ such that $\sum \lambda_i=0$ in $\Z/2\Z$ where $\lambda_i\in \Z/2\Z$ and $x_i\in \sm$. In other words, $G$ is the group $U$-scheme of even cardinality subsets of reduced divisors in $\sm$.

\begin{prop}\label{pseudo torsor}
Let $g^r_d$ be a linear system and consider the half $W^i\subset \td{\sm}^{(d)}$ of the variety of divisors $W$ lying over $g^r_d$. If we denote by $U_0:= U\cap g^r_d$ and $G_0: =G|_{U_0}$, then there is a $G_0$-action $\mu$ on $W^i|_{U_0}$ over $U_0$, making it a pseudo $G_0$-torsor on $U_0$ i.e. the induced morphism $(\mu,pr_2) :G_0\times_{U_0} W^i|_{U_0}\to  W^i|_{U_0}\times_{U_0} W^i|_{U_0}$ is an isomorphism.
\end{prop}
\begin{proof}
To simplify notation we denote $\td{U}:= \td{\sm}^{(d)}|_U$. The construction of the group action by $G_0$ is similar to the multiplication map defined in Proposition \ref{group scheme}. We first define a group action $G'\times_U \td{U} \to \td{U}$. The involution action $\sigma: \overline{\sm}\times_{\sm}\td{\sm}\to \td{\sm}$ induces the map $\sigma^{(d)}:(\overline{\sm}\times_{\sm}\td{\sm})^{(d)}\to \td{\sm}^{(d)}$. The natural projection $pr_1: \overline{\sm}\times_\sm \td{\sm}\to \overline{\sm}$ induces the map $pr_1^{(d)}:(\overline{\sm}\times_{\sm}\td{\sm})^{(d)}\to \td{\sm}^{(d)}$. Then we get by universal property $t: (p^{(d)}\circ pr_1^{(d)})^{-1}(U)\to G'\times _U \td{U}$. Again, we can easily check that the map $t$ is bijective on closed points and $G'\times _U \td{U}$ is smooth and hence normal, the map $t$ is an isomorphism. We define the group action as 
\begin{equation*}
    G'\times _U \td{U} \xrightarrow{t^{-1}} (p^{(d)}\circ pr_1^{(d)})^{-1}(U) \xrightarrow{\sigma^{(d)}|_U}  \td{U}
\end{equation*}

This defines another group action by restricting to $G\subset G'$. Finally, by Remark \ref{involution by even card subsets}, we see that the restriction of the group action by $G$ to $G_0$ defines a group action 
\begin{equation*}
    \mu: G_0\times_ {U_0} W^i|_{U_0}\to W^i|_{U_0}
\end{equation*}
At closed points, the group action $\mu(\C):G_0(\C)\times_ {U_0(\C)} (W^i|_{U_0})(\C)\to (W^i|_{U_0})(\C)$ simply sends
\begin{equation*}
    \left(\sum^d_{j=1}(\lambda_j,x_j), \sum_{j=1}^d (w_j,x_j)  \right) \mapsto \sum_{j=1}^d (\sigma^{\lambda_j}(w_j),x_j)
\end{equation*}
over $\sum_{j=1}^dx_j\in U_0(\C)$, 
where $\sum_{j=1}^d (w_j,x_j) \in (W^i|_{U_0})(\C)$ with $\pi(w_j)=x_j$ and we denote by $\sigma^0= Id, \sigma^1=\sigma$. 

Also, it is clear that $\mu$ is simply transitive on closed points. Then it follows by the normality of $ W^i|_{U_0}\times_{U_0} W^i|_{U_0}$ that the induced morphism $(\mu,pr_2): G_0\times_{U_0} W^i|_{U_0}\to  W^i|_{U_0}\times_{U_0} W^i|_{U_0}$ is an isomorphism.
\end{proof}

\begin{example}\label{linear system of canonical divisors}
Consider the linear system $|K_\sm|$ i.e. the linear system of canonical divisors. In this case, we have $d= 2g-2$ and $r=g-1$. Observe that $\textrm{dim}(W) = \textrm{dim}(\Pr^i)$ and the fiber of the morphisms $W^i \to \Pr^i$ at a point $[D]\in \Pr^i$ is $|D|.$ It can be shown (e.g. \cite[Section 6]{mumford74prym}, \cite{mumford71theta}) that: 
\begin{enumerate}
    \item $\td{\a}|_{W^1}: W^1 \to \Pr^1$ is birational.
    \item $\td{\a}|_{W^0}: W^0\to \Pr^0$ maps onto a divisor $\Theta \subset \Pr^0$ and is generically a $\P^1$-bundle. 
\end{enumerate}
\end{example}

\section{Modules over the sheaf of even Clifford algebras  }
\subsection{Conic fibrations and sheaves of even Clifford algebras }\label{conic fibrations}
For simplicity, we will only discuss conic fibrations i.e. flat quadric fibrations of relative dimension 1. A conic fibration $\pi:Q\to S$ over a smooth variety $S$ is defined by a rank 3 vector bundle $F$ on $S$, together with an embedding of a line bundle $q: L\to S^2F^\vee$ which is also thought of as a section in $S^2F^\vee\otimes L^\vee$. Then $Q$ is embedded in $\P(F) = \textrm{Proj}(\bigoplus_{k\geq 0} S^kF^\vee)$ as the zero locus of $q\in H^0(S,S^2 F^\vee\otimes L^\vee)  = H^0(\P(F), \O_{\P(F)/S}(2)\otimes (\pi')^*L^\vee)$ where we denote by $\pi':\P(F)\to S$ the projection morphism. The morphism $\pi:Q\to S$ obtained by restricting $\pi'$ to $Q$ is flat as $q:L\to S^2F^\vee$ is assumed to be an embedding and $S$ is smooth.

Given a conic fibration $\pi:Q\to S$, we define the sheaf of even Clifford algebras by following the approach of \cite{auel14quadrics}\footnote{Note that we write a line bundle-valued quadratic form as $q: L\to S^2F^\vee$ where the authors in \cite{auel14quadrics} write it as $L^\vee\to S^2F^\vee$.}. Note that $q:L\to S^2F^\vee$ induces an $\O_S$-bilinear map $q: F\times F\to L^\vee$ (again denoted by $q$). Then we can consider the two ideals $J_1$ and $J_2$ of the tensor algebra $T^\bullet (F\otimes F\otimes L )$ which are generated by
\begin{equation}
    v\otimes v \otimes f - \langle q(v,v),f\rangle , \quad u\otimes v\otimes f\otimes v\otimes w\otimes g- \langle q(v,v),f\rangle u\otimes w\otimes g
\end{equation}
respectively, where the sections $u,v,w\in F$ and $ f,g\in L$, and $\langle \cdot, \cdot \rangle: L^\vee\otimes L\to \O_S$ is the natural map. Then the even Clifford algebra is defined as the quotient algebra
\begin{equation}
    \m B_0:= T^\bullet (F\otimes F\otimes L ) /(J_1+J_2).
\end{equation}
The sheaf of algebra has naturally a filtration 
\begin{equation}
    \O_S= F_0 \subset F_1 = \m B_0
\end{equation}
obtained as the image of the truncation of the tensor algebra $T^{\leq i}(F\otimes F \otimes L) $ in $\m B_0$. Moreover, the associated graded piece $F_1/F_0 \cong \wedge^2 F \otimes L$. As an $\O_S$-module, we actually have $\m B_0\cong \O_S \oplus (\wedge^2F\otimes L)$ which can be seen by defining the splitting $\wedge^2F \otimes L \to F\otimes F\otimes L \to T^\bullet (F\otimes F\otimes L)/(J_1+J_2)$  where $\wedge^2 F$ is thought of as a subbundle of antisymmetric 2-tensors of $F\otimes F$. 

Now, given a conic fibration $\pi:Q\to S$ and its associated sheaf of even Clifford algebras $\m B_0$ as defined above, a $\m B_0$-module is a coherent sheaf on $S$ with a left $\m B_0$-module structure. We denote by $\Coh(S, \m B_0)$ the abelian category of $\m B_0$-modules on $S$. 

\subsection{Root stacks}
The main objects in this paper are $\m B_0$-modules. In order to study the category of $\m B_0$-modules, it is easier to work with a root stack cover of $S$. The advantage is that the category of $\m B_0$-modules is equivalent to the category of modules over a sheaf of Azumaya algebras on the root stack. For more details about root stacks, we refer the readers to \cite{Cadman07rootstack}. 

Let $\m L$ be a line bundle on a scheme $X$ and $s\in H^0(X,\m L)$ and $r$ a positive integer. The pair $(\m L,s)$ defines a morphism $X\to \left[ \bb A^1/\bb G_m\right]$, and the $r$-th power maps on $\bb A^1$ and $\bb G_m$ induce a morphism $\theta_r : \left[ \bb A^1/\bb G_m\right]\to \left[ \bb A^1/\bb G_m\right]$. Following \cite{Cadman07rootstack}, we define the $r$-th root stack $X_{\m{L},s,r}$ as the fiber product 
\begin{equation*}
    X\times_{ \left[ \bb A^1/\bb G_m\right],\theta_r} \left[ \bb A^1/\bb G_m\right].    
\end{equation*}

The $r$-th root stack $X_{\m L,s,r}$ is a Deligne-Mumford stack. Locally on $X$, when $\m L$ is trivial, $X_{\m L,s,r}$ is just the quotient stack $\left[ \Spec\left(\O_X[t]/(t^r-s)\right)/\mu_r\right]$ where $\mu_r$ is the group of $r$-th roots of unity acting on $t$ by scalar action. The root stack $X_{\m L,s,r}$ has $X$ as its coarse moduli space. There is a tautological sheaf $\m T$ on $X_{\m L,s,r}$ satisfying $\m T^{\otimes r}\cong \psi^*\m L$ where $\psi: X_{\m L,s,r} \to X$ is the projection. When the zero locus of $s$ is connected, every line bundle on $X_{\m L,s,r}$ is isomorphic to $\psi^*G\otimes \m T^{\otimes k}$ where $k\in\{0,...,r-1\}$ is unique and $G$ is unique up to isomorphism \cite[Corollary 3.1.2]{Cadman07rootstack}. For our purposes, we will mainly consider the case $\m L=\O_X(D)$ for an effective Cartier divisor $D$ and $s=s_D$ the section vanishing at $D.$ In this case, we will simply write $X_{\O_X(D), s_D,r} = X_{D,r}$ and the tautological sheaf $\m T$ as $\O(\frac{D}{r})$.

Similarly, it is pointed out in \cite[Lemma 2.1.1]{Cadman07rootstack} that there is an equivalence of categories between the category of morphisms $X\to \left[\bb A^n/\bb G_m^n\right] $ and the category whose objects are $n$-tuples $(\m L_i,t_i)^n_{i=1}$, where $\m L_i$ is a line bundle on $X$ and $t_i\in H^0(X,\m L_i)$, and whose morphisms $(\m L_i,t_i)^n_{i=1}\to (\m L'_i,t'_i)^n_{i=1}$ are $n$-tuples $(\varphi_i)^n_{i=1}$ where $\varphi_i:\m L_i\to \m L_i'$ is an isomorphism such that $\varphi_i(t_i)=t'_i$. If we let $\bb D:= (D_1,...,D_n)$ be an $n$-tuple of effective Cartier divisors and $\vec{r} = (r_1,...,r_n)$, then the $n$-tuples $(\O_X(D_i),s_{D_i})^n_{i=1}$ will determine a morphism $X\to \left[\bb A^n/\bb G_m^n\right].$
Also, the morphisms on $\bb A^n$ and $\bb G^n_m$ sending $(x_1,...,x_n)\mapsto (x_1^{r_1},...,x_n^{r_n})$ induce a morphism $\theta_{\vec{r}}:\left[\bb A^n/\bb G_m\right]\to \left[\bb A^n/\bb G_m\right]$. We define $X_{\bb D, \vec{r}}$ as the fiber product
\begin{equation*}
    X\times _{ \left[ \bb A^n/\bb G^n_m\right],\theta_{\vec{r}}} \left[ \bb A^n/\bb G^n_m\right].
\end{equation*}

This can be interpreted as iterating the $r$-th root stack construction for $n=1$. There are the tautological sheaves $\O\left(\frac{D_i}{r_i}\right)$ on $X_{\bb D,\vec{r}}$ satisfying $\O\left(\frac{D_i}{r_i}\right)^{\otimes r_i}\cong \psi^*\O_X(D_i)$. Every line bundle $F$ on $X_{\bb D,\vec{r}}$ can be written as 
\begin{equation*}
     F \cong \psi^* G \otimes \prod^r _{i=1} \m \O\left(\frac{D_i}{r_i}\right)^{\otimes k_i}
\end{equation*}
where $0\leq k_i \leq r_i-1$ are unique, and $G$ is unique up to isomorphism and $\psi:X_{\bb D,\vec{r}}\to X $ is the projection \cite[Corollary 3.2.1]{Cadman07rootstack}.

\begin{lem} \label{iterated_root_stacks}
Let $D= D_1+ ... +D_n$ where $D_i$ are pairwise disjoint effective Cartier divisors. If $r= r_1 =... = r_n$, we have 
\begin{equation*}
    X_{\bb D,\vec{r}}\xrightarrow{\sim} X_{D,r}.
\end{equation*}
\end{lem}
\begin{proof}
An object of $X_{D,r}$ over a scheme $T$ consists of a quadruple $(f,N,t,\varphi)$ where $f:T\to X$ is a morphism, $N$ a line bundle, $t\in H^0(T,N)$ and $\varphi:N^{\otimes r} \xrightarrow{\sim} f^* \O(D)$ is an isomorphism such that $\varphi(t^{\otimes r})= f^*s$ and $s$ is the section of $\O_X(D)$ vanishing at $D$. 

On the other hand, an object of $X_{\bb{D},\vec{r}}$ consists of $(f ,(N_i)_{i=1}^n, (t_i)^n_{i=1},(\varphi_i)^n_{i=1})$ where $f:T\to X$ a morphism, $N_i$ is a line bundle, $t_i\in H^0(T,N_i)$ and $\varphi_i:N_i^{\otimes r_i} \xrightarrow{\sim} f_i^*\O(D_i)$ is an isomorphism such that $\varphi_i(t_i^{\otimes r_i})= f_i^*s_i$ and $s_i $ is the section of $\O_X(D_i)$ vanishing at $D_i$. We see that there is a natural morphism $\a: X_{\bb D,\vec{r}}\to X_{D,r}$ over $X$ sending 
\begin{equation}\label{mapsofrootstacks}
    (f ,(N_i)_{i=1}^n, (t_i)^n_{i=1},(\varphi_i)^n_{i=1}) \mapsto  (f, \bigotimes_{i=1}^n N_i, \bigotimes_{i=1}^n t_i, \bigotimes_{i=1}^n \varphi_i).
\end{equation}
To see that this is an isomorphism, we restrict to each open neighborhood $U_i$ of $D_i$ away from $D_j$ ($j\neq i$) such that $\O(D_j)|_{U_i}\xrightarrow{\sim}\O$ for $j\neq i$ and $\O(D_i)|_{U_i}\xrightarrow{\sim} \O(D_1+..+D_n)|_{U_i}$. Then it is clear that the functor (\ref{mapsofrootstacks}) over $U_i$ is essentially surjective i.e. the image of the quadruples $(f ,(N_i)_{i=1}^n, (t_i)^n_{i=1},(\varphi_i)^n_{i=1})$ where $N_j\cong \O$ and $t_j=1$ for $j\neq i$ is dense.

\end{proof} \label{equivariant module}

\begin{example}\label{root stack over affine scheme} (Affine scheme)
Let $X=\Spec(R)$, $L=\O_X$ and $s$ be a section of $\O_X$. Then $X_{L,s,r}\cong \left[\Spec R'/\mu_r\right]$, where $R'= R[t]/(t^r-s)$, and $\gamma\cdot t= \gamma^{-1}t$ and $\gamma\cdot a= a$ for $a\in R$ and $\gamma\in \mu_r$. A quasi-coherent sheaf on $\left[ \Spec R'/\mu_r\right]$ is a $R'$-module $M$ with a $\mu_r$-action on $M$ such that for $\gamma\in \mu_r, b\in R', m \in M$, we have 
\begin{equation*}
    \gamma \cdot (b\cdot m) = (\gamma \cdot b)\cdot (\gamma\cdot m).
\end{equation*}
As $\mu_r$ is diagonalizable, there is a $\Z/r\Z$-grading $M\cong M_0\oplus...\oplus M_{r-1}$ where $\gamma \cdot m_i = \gamma^{i}m_i$ for $m_i\in M_i$. Note that the components are indexed by the group of characters of $\mu_r$, which is $\Z/r\Z$. Similarly, $R' \cong R'_0\oplus...\oplus R'_{r-1}$ where $R'_0=R$. In particular, we see that $\gamma: M\to M$ is an $R$-module homomorphism, and so each $M_i$ is an $R$-module. 

\end{example}

\begin{example}\label{root stack = quotient stack} (Cyclic cover)
When there exists a line bundle $N$ such that $f:N^{\otimes r}\cong \m L $, we can take the cyclic cover for section $s$, defined as 
\begin{equation*}
    \phi: \td{X}: = \underline{\Spec}\left(\m A_X\right)\to X, \qquad \m A_X:= \O_X\oplus N^\vee\oplus... \oplus (N^\vee)^{r-1}
\end{equation*}
where the algebra structure of $\m A_X$ is given by the map $(N^\vee)^{\otimes r}\xrightarrow{f^\vee} (\m L)^\vee\xrightarrow{s^\vee}\O $. By \cite[Th\'{e}or\`{e}me 3.4]{borne07rootstacks}, we know that 
\begin{equation*}
    \left[ \td{X}/\mu_r\right] \cong X_{\m L, s,r}.
\end{equation*}

Suppose $X$ is a smooth curve and $D= p_1+...+p_k$ is a reduced divisor and $r=2$. The cyclic cover $\phi:\td{X}\to X$ is branched at $p_i$, we denote by $w_i$ the ramification points such that $\phi(w_i)=p_i$. Note that the points $w_i$ are also the fixed points under the involution of $\td{X}$. In this case, the line bundles on the root stack $X_{D,2}$ can be described in terms of line bundles on $\td{X}$ as follows. 

Since the root stack $X_{D,2}$ is the quotient stack $\left[\td{X}/\mu_2\right]$, a line bundle on $\left[\td{X}/\mu_2\right]$ is the same as a $\mu_2$-equivariant line bundle on $\td{X}$. On $\O_{\td{X}}(w_i)$, there is a group action on $\textrm{Tot}(\O(w_i))$ which fixes the canonical section vanishing at $w_i$, we will denote by $L(w_i)$ the line bundle $\O(w_i)$ together with this $\mu_2$-equivariant sheaf structure. In particular, the induced $\mu_2$-action on the fiber of $\O(w_i)$ is $-Id.$ The pull back of a line bundle $F$ on $X$ to $\td{X}$ is equipped with a natural $\mu_2$-equivariant sheaf structure, whose induced action on the fiber at $w_i$ is $Id$ and the $\mu_2$-equivariant bundle is again denoted by $\phi^* F$. Since $\phi^*N\cong \O\left(\sum_i w_i\right)$, we can write 
        \begin{equation*}
            \O(w_i) \cong \O\left(2w_i +\sum_{i\neq j} w_j\right) \otimes \phi^*N^{\vee} \cong \O\left(\sum_{i\neq j} w_j\right) \otimes \phi^*(N^\vee\otimes \O(p_i)).
        \end{equation*}
        So we see that $L\left(\sum_{i\neq j} w_j\right) \otimes \phi^*(N^\vee\otimes \O(p_i))$ has the same underlying line bundle as $L(w_i).$ 

        As discussed above, every line bundle on $X_{D,2}$ is of the form $\psi^*F\otimes \O\left(\sum_{i\in I}\frac{p_i}{2}\right)$. In terms of the language of $\mu_2$-equivariant line bundles, we see that $\O\left(\frac{p_i}{2}\right)$ on $X_{D,2}$ corresponds to $L(w_i)$ on $\td{X}$. Moreover, the pushforward $\psi_*\wh{E}$ of a vector bundle $\wh{E}$ on $X_{D,2}$ is the $\mu_2$-invariant subbundle of the $\mu_2$-equivariant bundle $\phi_*\td{E}$, denoted by $(\phi_*\td{E})^{\mu_2}$, where $\td{E}$ is the $\mu_2$-equivariant vector bundle corresponding to $\wh{E}$.

\end{example}

\begin{prop}(\cite[Proposition 3.12]{borne07rootstacks}) \label{locally direct sum}
Suppose that $div(s)$ is an effective Cartier divisor. Let $\m F$ be a locally free sheaf on $X_{\m L,s,r}$. For each point $x\in X$, there exists a Zariski open neighborhood $U$ of $x$ such that $\m F|_{\psi^{-1}(U)}$ is a direct sum of invertible sheaves. 
\end{prop}

\subsection{Root stack associated to a conic fibration}

Let $\pi:Q\to S$ be a conic fibration as defined in Section \ref{conic fibrations}. We denote by $S_1\subset S$ the discriminant locus of degenerate conics. A conic fibration $\pi:Q\to S$ is said to have simple degeneration if all the fibers are quadrics of corank $\leq 1$ and the discriminant locus is smooth. In other words, the degenerate fibers cannot be double lines. For the rest of the paper, we will assume all the conic fibrations to have simple degeneration.  

We define the $2$-nd root stack of $S$ along $S_1$ as $\wh{S}: = S_{S_1,2}$ and $\psi:\wh{S}\to S$ the projection. Then it is shown in \cite[Section 3.6]{kuznetsov2008derived} that there is a sheaf of algebra $\wh{\m B}_0$ on $\wh{S}$ such that $\psi_*\wh{\m B}_0 = \m B_0$, so there is an equivalence of categories
\begin{equation}
    \psi_* : \Coh (\wh S,\wh{\m B}_0) \xrightarrow{\sim } \Coh(S, \m B_0) 
\end{equation}
where $\Coh(\wh S,\wh{\m B}_0)$ is the abelian category of coherent sheaves on $\wh S$ with a left $\wh{\m B}_0$-module structure.
Moreover, the sheaf of algebra $\wh{\m B}_0$ is a sheaf of Azumaya algebra. 

Suppose $C\subset S$ is a smooth curve such that the intersection $S_1\cap C$ is tranverse, we restrict the conic fibration $Q\to S$ to the smooth curve $C\subset S$ to get a conic fibration $\pi|_C:Q|_C\to C$ with simple degeneration. We get the root stack $\wh{C}:=C_{S_1\cap C,2}\cong \wh{S}|_C$ and denote by $\wh{\m B}_0$ the restriction $\wh{\m B}_0|_{\wh{C}}$ by abuse of notation. The sheaf of algebra $\wh{\m B}_0$ on $\wh{C}$ is a trivial Azumaya algebra since $\wh{C}$ is smooth and $\dim(\wh{C})=1$ \cite[Corollary 3.16]{kuznetsov2008derived}. That means there exists a rank 2 vector bundle $E_0$ on $\wh{C}$ (root stack construction is preserved under pull back) such that $\wh{\m B}_0\cong \m End(E_0)$ and it induces the equivalence of categories:
\begin{equation*}
	 \left.\begin{array}{rcccl}
	 \Coh(\wh{C})&\xrightarrow{\sim}&\Coh(\wh{C},\wh{\m B}_0) &\xrightarrow{\sim}&\Coh(C, \m B_0) \\
	 \m F&\longmapsto& \m F\otimes E_0&\longmapsto& \psi_*(\m F\otimes E_0)
	 \end{array}\right.
	\end{equation*}
	Let us define the rank of a $\m B_0$-module to be the rank of the underlying $\O_C$-module. In particular, by the equivalence of categories, we have the following: 
\begin{cor}\label{rootstack over curves}
 The rank of a $\m B_0$-module $\psi_* (\m F \otimes E_0)$ on $C$ must be a multiple of 2. 
\end{cor}

Let $p\in S_1\cap C$. According to Proposition \ref{locally direct sum}, there exists a Zariski open neighbourhood $U$ of $p$ such that $E_0|_{\psi^{-1}(U)}\cong L_1\oplus L_2$ for some line bundles $L_i$ on $\psi^{-1}(U).$ As explained in Example \ref{root stack over affine scheme} and by shrinking $U$ if necessary, we can assume that $U= \Spec(R)$ is an affine neighbourhood of $p$ and consider the double cover $\td{U}= \Spec(R')$ where $R':=R[t]/(t^2-s)$ and $w= div(s)$ maps to $p$. Then the root stack restricted over $U$ is simply $\wh{U}= \left[ \left(\Spec(R[t]/(t^2-s)\right)/\mu_2\right]$. Then, as explained in Example $\ref{root stack = quotient stack}$, each $L_i$ is an $\mu_2$-equivariant bundle on $\td{U}$, and in particular, each $L_i$ defines a character $\chi_{i,w}:\mu_2\to \C^*$ of $\mu_2$ at the fiber $L_i|_w$. 

\begin{prop}\label{characters of E0}
We have $\chi_{1,w} (-1) \cdot \chi_{2,w}(-1) = -1.$ In particular, the $\mu_2$-invariant part of the fiber $E_0|_{w}$ is one-dimensional. 
        \end{prop}
        \begin{proof}
        We will follow the notations in the preceding paragraph. We can further reduce to the localization of $R$ at $p$, we will again write the local ring as $R$.

      As explained in Example \ref{equivariant module}, the rank 2 vector bundle $E_0$ on $\wh{U}$ is an $R'$-module $M$ with a $\mu_2$-action such that $M \cong M_0\oplus M_1$ where $M_i$ are $R$-modules. By Proposition \ref{locally direct sum}, we can write
      $M\cong R'[l_1]\oplus R'[l_2]$ as $\Z/2\Z$-graded $R'$-modules where $l_i\in \{0,1\}$, or equivalently, choose $e_1\in M_{l_1}$ and $e_2\in M_{l_2}$ such that $ M\cong R'e_1\oplus R'e_2$. Since $\chi_{i,w}(-1) = (-1)^{l_i}$ for $i=1,2$, it suffices to check $l_1+l_2=1\in \Z/2\Z$.
      
        Suppose the contrary that $l_1=l_2=0$ (or $l_1=l_2=1$). Recall that $E_0$ satisfies $\psi_*\m End(E_0) \cong \m B_0$ as sheaf of algebras. Since the conic of $Q$ over $p$ is degenerate, its even Clifford algebra $\m B_0|_p$ is not isomorphic to the endomorphism algebra of rank 2. 
        
        On the other hand, there is a natural morphism 
        \begin{equation*}
           \a: \psi_*\m End(E_0)\to \m End(\psi_*E_0).
        \end{equation*}
         Since $E_0$ corresponds to a $\Z/2\Z$-graded $R'$-module, $\m End(E_0)$ also corresponds to a $\Z/2\Z$-graded $R'$-module and so $\psi_*\m End(E_0)$  corresponds to the $\mu_2$-invariant part i.e. $\left(\m End(E_0)\right)_0$ which is an $R$-module. In terms of the $R'$-basis $\{e_1,e_2\}$, $\left(\m End(E_0)\right)_0$ consists of the homogeneous $R$-module homomorphisms $\delta$ of degree 0:
     \begin{align*}
         e_1&\mapsto u_0e_1 + u_0'e_2\\
         e_2&\mapsto v_0e_1+ v_0'e_2
     \end{align*}
     where $u_0,u_0',v_0,v_0'\in R'_0 = R$. Similarly, the module $\psi_*E_0$ is the $\mu_2$-invariant part of $R'e_1\oplus R'e_2$ which is freely generated by $\{f_1=e_1,f_1=e_2\}$ (or $\{f_1=te_1,f_2=te_2\}$ when $l_1=l_2=1$) as $R$-modules. For $l_1=l_2=0$, $\delta\in \psi_*\m End(E_0)$ is mapped to an image in $\m End(\psi_*E_0)$ of the form 
     \begin{align*}
         f_1&= e_1\mapsto u_0e_1 + u_0'e_2=u_0f_1 + u_0'f_2\\
         f_2&= e_2\mapsto v_0e_1+ v_0'e_2=v_0f_1+ v_0'f_2
     \end{align*}
     Since $u_0,v_0,u_0',v_0'$ are arbitrary elements in $R$, the image of $\a$ will be the endomorphism algebra over $R$ i.e. $\a$ is an isomorphism of $R$-algebras, which is a contradiction. For $l_1=l_2=1$, the image of $\a$ is also surjective for the same reason.
    
    The second assertion is clear because $\chi_{1,w} (-1) \cdot \chi_{2,w}(-1) = -1$ means that one and only one of $L_1|_w$ and $L_2|_w$ is $\mu_2$-invariant. 
        \end{proof}

\subsection{Moduli space of $\m B_0$-modules}
In order to guarantee the existence of a moduli space of $\m B_0$-modules, we will use Simpson's theory of moduli spaces of $\Lambda$-modules \cite{Simpson94moduli1}. Let us recall the definition of a sheaf of rings of differential operators from Simpson's paper \cite{Simpson94moduli1} and follow its notations closely. Suppose $S$ is a noetherian scheme over $\C$, and let $f:X\to S$ be a scheme of finite type over $S.$  A sheaf of rings of differential operators on $X$ over $S$ is a sheaf of (not necessarily commutative) $\O_X$-algebras $\Lambda$ over $X, $ with a filtration $\Lambda_0\subset \Lambda_1\subset...$ which satisfies the following properties:
        \begin{enumerate}
            \item $\Lambda =\bigcup^\infty_{i=0}\Lambda_i $ and $\Lambda_i\cdot \Lambda_j\subset \Lambda_{i+j}.$
            \item The image of the morphism $\O_X\to \Lambda$ is equal to $\Lambda_0.$
            \item The image of $f^{-1}(\O_S)$ in $\O_X$ is contained in the center of $\Lambda.$
            \item The left and right $\O_X$-module structures on $Gr_i(\Lambda) := \Lambda_i/\Lambda_{i-1}$ are equal. 
            \item The sheaves of $\O_X$-modules $Gr_i(\Lambda)$ are coherent.
            \item The sheaf of graded $\O_X$-algebras $Gr(\Lambda) :=  \bigoplus^\infty_ {i=0}Gr_i(\Lambda)$ is generated by $Gr_1(\Lambda)$ in the sense that the morphism of sheaves 
            \begin{equation*}
                Gr_1(\Lambda)\otimes_{\O_X}....\otimes_{\O_X} Gr_1(\Lambda) \to Gr_i(\Lambda)
            \end{equation*}
            is surjective. 
        \end{enumerate}
        The definition of stability condition for a $\Lambda$-module $\m E$ is similar as the case of coherent sheaves. We define $d(\m E), p(\m E,n), r(\m E)$ to be the dimension, Hilbert polynomial and rank of the underlying coherent sheaf of $\m E$ respectively. As defined in \cite{Simpson94moduli1}, a $\Lambda$-module $\m E$ is $p$-semistable (resp. $p$-stable) if it is of pure dimension, and if for any sub-$\Lambda$-module $\m F\subset \m E $ with $0<r(\m F)<r(\m E)$, there exists an $N$ such that 
        \begin{equation}
            \frac{p(\m F, n)}{r(\m F)}\leq \frac{p(\m E,n)}{r(\m E)}
        \end{equation}
        (resp. $<$) for $n\geq N$. 
        
        We will now specialize to the case where $S= \Spec(\C)$ and $X=\P^2$. We fix a conic fibration $\pi:Q\to \P^2$ for the rest of the section and let $\m B_0$ be the associated sheaf of even Clifford algebras. 
        \begin{prop}
        The sheaf of $\O_{\P^2}$-algebra $\m B_0$ is a sheaf of rings of differential operators. 
        \end{prop}
        \begin{proof}
          Recall that as an $\O_{\P^2}$-module, $\m B_0\cong \O_{\P^2}\oplus (\wedge^2 F\otimes \m L)$ with the filtration $\Lambda_0 = \O_{\P^2}, \Lambda_i = \m B_0$ for $i\geq 1$. Properties (1), (2), and (5) are clearly satisfied. The center of $\m B_0$ is $\Lambda_0$, so (3) is also satisfied. The left and right $\O_{\P^2}$-module on $\m B_0$ coincide by definition, so the induced left and right $\O_{\P^2}$-module structure  also coincide on $Gr_i(\Lambda).$ Finally, since $Gr_i(\Lambda)=0$ for $i>1$, property (6) is satisfied trivially. 
        \end{proof}
        
        Since $\m B_0$ is a sheaf of rings of differential operators, \cite[Theorem 4.7]{Simpson94moduli1}  guarantees the existence of a moduli space of semistable $\m B_0$-modules with a fixed Hilbert polynomial whose closed points correspond to Jordan equivalence classes of $\m B_0$-modules. Note that specifying a Hilbert polynomial is equivalent to specifying the Chern character for the case of $\P^2$. In this paper, we will be primarily interested in the moduli space of semistable $\m B_0$-module on $\P^2$ with Chern character $(0,2d,e)$, denoted by $\f M_{d,e}$. The Chern character of a $\m B_0$-module is defined as the Chern character of the underlying coherent sheaf. 
        
        Recall that in the case of a moduli space of one-dimensional coherent sheaves, one can define a support morphism by following Le Potier's construction \cite[Section 2.2]{lepotier1993}: a pure dimension one coherent sheaf $G$ on a smooth projective polarized surface $(Y, \O_Y(1))$ is Cohen-Macaulay, so $G$ admits a resolution by vector bundles of rank $r$:
        \begin{equation*}
            0\to A\xrightarrow{u} B\to G\to 0
        \end{equation*}
        Then we can define the so-called Fitting support $\textrm{Fit}(G)$ to be the vanishing subscheme of the induced morphism $\wedge^r u$. It can be checked that $\textrm{Fit}(G)$ is independent of the resolution and $\textrm{Fit}(G)$ represents its first Chern class $\textrm{c}_1(G)$. As the Fitting support construction works in families, it defines a support morphism from the moduli space of pure dimension one sheaves to the Hilbert scheme of curves of degree $\mathrm{c}_1(G)\cdot \O_Y(1)$. When $Y=\P^2$, the latter space is simply the linear system $|\O_{\P^2}(\textrm{deg}(G))| $.
        
        Similarly, for a pure dimension one $\m B_0$-module $M$ of Chern character $(0,2d,e)$, we can define the Fitting support of its underlying coherent sheaf $\textrm{Fit}(M): =\textrm{Fit}(\textrm{Forg}(M))$. This construction again works in families and induces a support morphism on the moduli space of $\m B_0$-modules
        \begin{equation*}
            \Upsilon: \f M_{d,e} \to |\O(2d)|.
         \end{equation*}
       
        The following is observed in \cite{Lahoz15ACM}.       
        \begin{prop}[\cite{Lahoz15ACM}]
        The support morphism $\Upsilon: \f M_{d,e}\to |\O(2d)|$ factors through $|\O(d)|\subset |\O(2d)|.$ 
        \end{prop}
        \begin{proof}
         By Corollary \ref{rootstack over curves}, if a $\m B_0$-module $M$ is (set-theoretically) supported on a smooth curve $C$ of degree $2d$, its rank must be a multiple of 2. It follows that the Fitting support $\textrm{Fit}(M)$ must be a nonreduced curve. So we see that the fiber over a smooth curve $C\in |\O(2d)|$ is empty. 
         
         If a $\m B_0$-module $M$ is supported on an integral but singular curve $C$ of degree $2d$, we can pull back $M$ to its normalization and apply Corollary \ref{rootstack over curves}, the rank of the pull back of $M$ must be a multiple of 2. Hence, $M$ has generically rank multiple of 2, its Fitting support will not be reduced. So again the fiber over $C\in |\O(2d)|$ is empty.  
         
         Finally, for a $\m B_0$-module $M$ supported on a reduced but reducible curve, the same argument as in the previous paragraph shows that the fiber over such a curve is empty. 
         
         Hence, the image of $\Upsilon$ must be contained in the nonreduced locus $|\O(d)|\subset |\O(2d)|.$
        \end{proof}
        
        \begin{rmk}
        From now on, we will write $\Upsilon: \f M_{d,e}\to |\O(d)|$. 
        \end{rmk}

        \begin{thm}[\cite{Lahoz15ACM}] \label{fibers as even card subsets}
         Let $U_d\subset |\O_{\P^2}(d)| $ be the open subset of smooth degree $d$ curves which intersect $\Delta$ transversally and $C\in U_d$. Then 
        \begin{equation*}
            \Upsilon^{-1} (C) \cong \bigsqcup_{I} \Pic^{-|I|/2}C.
        \end{equation*}
        where $I$ runs over the even cardinality subsets of $\{1,...,dk\}$ and $k:=\deg(\Delta)$.  
        \end{thm}
        \begin{proof}
         We will recall the description of the fiber $\Upsilon^{-1}(C)$ in \cite[Theorem 2.12]{Lahoz15ACM} and provide more details as it will be important for our purposes in later sections. A $\m B_0$-module $M\in \Upsilon^{-1}(C)$ is a rank 2 vector bundle supported on $C$, so we can restrict our attention to $\m B_0$-modules on $C.$ Note that a $\m B_0$-module $M$ that is a rank 2 vector bundle on $C$ is automatically $p$-stable, since the rank of any $\m B_0$-module on $C$ must be a multiple of $2$.
         
        
        As explained in previous section, there is a rank 2 vector bundle $E_0$ on the 2nd-root stack $\wh{C}:= C_{C\cap \Delta, 2}$ and an equivalence of categories $\psi_*:\Coh(\wh{C}, \m End(E_0) ) \xrightarrow{\sim} \Coh(C, \m B_0 )$ where $\psi:\wh{C}\to C$ is the projection morphism. That means we are looking for line bundles $\wh{L}$ on $\wh{C}$ such that $\ch(i_*\psi_*(E_0\otimes \wh{L}))= (0,2d,e)$ where $i:C\to \P^2$ is the inclusion map. It is clear that $\ch_0(i_*\psi_*(E_0\otimes \wh{L}))=0$ and $\ch_1(i_*\psi_*(E_0\otimes \wh{L}))=2d.$ To compute $\ch_2$, we use the fact that is easily computed by the Grothendieck-Riemann-Roch theorem:
        \begin{equation}
            \ch_2(i_*G)  = \deg(G) - \frac{d^2}{2}\rk(G)
        \end{equation}
        for a vector bundle $G$ on $C.$ So it is equivalent to finding all $\wh{L}$ on $\wh{C}$ such that $e = \deg(\psi_*(E_0\otimes \wh{L})- d^2  $ or $\deg(\psi_*(E_0\otimes \wh{L})) = e+d^2.$
        
        \noindent \textbf{Case 1}: When $k=\deg(\Delta)$ is even, in which case $\O_{\P_2}(k)|_C$ admits a square root $\O_{\P_2}(k/2)|_C$, we can take the the cyclic cover $\phi: \td{C}\to C $ of order 2 branched at $C\cap \Delta$ with an involution action. As explained in Example \ref{root stack = quotient stack} the root stack $\wh{C}$ is isomorphic to the quotient stack $\left[\td{C}/\mu_2\right]$. Moreover, the morphism $\phi: \td{C}\to C$ factors as $\td{C} \xrightarrow{\eta} \wh{C} \xrightarrow{\psi} C$. 
        
        Let $w_i\in \td{C}, p_i=\phi(w_i)\in C\cap \Delta $ be the ramification and branch points respectively. Recall that a line bundle $\wh{L}$ on $\wh{C}$ can be written as $\psi^*J \otimes \O\left(\sum \lambda_i\frac{p_i}{2}\right)$ such that $J$ is a line bundle on $C$, where $\lambda_i\in \{0,1\}$. As a $\mu_2$-equivariant line bundle $\wh{L} = \phi^*(J)\otimes L\left(\sum \lambda_i w_i\right)$ on $\td{C}$ (following the notation in Example \ref{root stack = quotient stack}), 
        \begin{equation}\label{eq:chern class}
        \begin{aligned}
            \c_1\left(\phi_* \left(E_0\otimes \phi^*(J)\otimes L\left(\sum \lambda_i w_i\right) \right)^{\mu_2}  \right) &= \c_1\left(J\otimes \phi_* \left(E_0\otimes L\left(\sum \lambda_i w_i\right) \right)^{\mu_2}\right)\\
            &= 2\c_1(J) + \c_1\left(\phi_* \left(E_0\otimes L\left(\sum \lambda_i w_i\right) \right)^{\mu_2}\right)
        \end{aligned}
        \end{equation}
        Note that we have the short exact sequence
        \begin{equation*}
            0\to \phi_*\left(E_0 \right)^{\mu_2} \to \phi_* \left( E_0 \otimes L\left(\sum \lambda_i w_i\right)\right)^{\mu_2} \to \bigoplus_{i\in I} \phi_*\left(E_0\otimes  L\left(\sum \lambda_i w_i\right)\otimes\O_{w_i}\right) ^{\mu_2}\to 0
        \end{equation*}
        where $I$ is the subset of $\{1,2,...,dk\}$ such that $\lambda_i=1$ for $i\in I$.
        Since the dimension of the fiber $\phi_*\left(E_0\otimes \O_{w_i}\right) ^{\mu_2}|_{p_i}\cong (E_0|_{w_i})^{\mu_2}$ at $p_i$ is 1 by Proposition \ref{characters of E0}, we have $\c_1\left(\phi_*\left(E_0\otimes \O_{w_i}\right) ^{\mu_2}\right)=1 $ which implies that
        $$\c_1\left(\phi_*\left(E_0\otimes \O_{w_i}\otimes L\left(\sum \lambda_i w_i\right)\right) ^{\mu_2}\right)=1.$$ Then the last expresssion of (\ref{eq:chern class}) becomes
        \begin{equation*}
            2\c_1(J) + \c_1\left(\phi_*(E_0)^{\mu_2}\right) + \left|I\right|.
        \end{equation*}
        where $|I|$ is its cardinality. Since $E_0$ on $\wh{C}$ is determined up to tensorization by a line bundle, this expression means that we can assume $\deg\left(\left(\phi_*E_0\right)^{\mu_2}\right) = e+d^2$. 
        
        We also see that the condition $\deg(\psi_*(E_0\otimes \wh{L})) = e+d^2$ becomes 
        \begin{equation*}
            e+d^2 = 2\deg(J)  + \deg(\psi_*(E_0)) + |I| \implies 0 = 2\deg(F) + |I|,
        \end{equation*}
        which is the same as saying that the degree of $\wh{L}$ as a line bundle on $\td{C}$ is 0. The condition $2\deg(F)+ |I|=0$ only makes sense if $|I|$ is even. We also see that for each fixed $I$, the set of line bundles satisfying the condition above is $\Pic^{-|I|/2}(C)$. Thus, $\Upsilon^{-1}(C)\cong \bigsqcup_I \Pic^{-|I|/2} (C)$ where $I$ runs over the set of even cardinality subsets of $dk$. 
        
        \noindent \textbf{Case 2}: 
        When $k=\deg(\Delta)$ is odd, we will use the trick by choosing an auxiliary line $H\subset \P^2$ which intersects $C$ transversally and $D_a:=H\cap C$ is disjoint from $D:=C\cap \Delta$. Then the line bundle $\O_C(D+D_a) \cong \O_{\P^2}(k+1)|_C$ has a natural square root $\O_{\P^2}((k+1)/2)|_C$, so we can again consider the cyclic cover $\td{C}$ branched at $C\cap (\Delta+H)$. The root stack $\overline{ C}:=C_{D+D_a,2}$ is now isomorphic to the quotient stack $\left[\td{C}/\mu_2\right]$. We again denote by $\wh{ C}$ the root stack $C_{C\cap \Delta,2}$. By Lemma \ref{iterated_root_stacks}, the stack $\overline{C}$ is isomorphic to  $C_{D,D_a,(2,2)}$ which is constructed as a fiber product, hence $C_{D, D_a,(2,2)}$ projects to $\wh{C}.$ We denote the composition by $f:\overline{C}\xrightarrow{\sim} C_{D, D_a,(2,2)}  \to \wh{C}$. 
        \begin{equation*}
            \begin{tikzcd} \td{C}\arrow[d,"q"]\arrow[rrd,"\phi"]&&\\
                    {\overline{ C}}\arrow[r,"f"]&\wh{ C}\arrow[r, "\psi"]& C
            \end{tikzcd}
        \end{equation*}
        
        Let $\wh{L}$ be a line bundle on $\wh{C}$, we want to find all such line bundles such that $\ch(\psi_*(E_0\otimes\wh{L})) = (0,2d,e)$. Recall that this is equivalent to finding $\deg(\psi_*(E_0\otimes\wh{L})) = e+d^2$. By \cite[Theorem 3.1.1 (3)]{Cadman07rootstack} (the proof there works for any vector bundle), we know that $\wh{M} \cong f_*f^*\wh{M}$ for any vector bundles on $\wh{C}$, so 
        \begin{align*}
           \psi_*(E_0\otimes \wh{L}) = \psi_*f_*\left(f^*\left(E_0\otimes \wh{L}\right)\right) = \phi_* \left(f^* \left(E_0\otimes \wh{L}\right)\right)^{\mu_2}
        \end{align*}
        As $\overline{C}\cong \left[ \td{C}/\mu_2\right]$, $f^*\left(E_0\otimes \wh{L}\right)$ on $\overline{C}$ is a $\mu_2$-equivariant vector bundle on $\td{C}$ whose induced $\mu_2$-characters at the fixed points $w_i\in D_a$ is trivial. In other words, the problem now is to find all line bundles on $\td{C}$ of the form $\phi^*(F)\otimes \O(\sum \lambda_i w_i)$ where $w_i\in \phi^{-1}(D)$ such that 
        \begin{equation*}
            \deg\left( \phi_*\left(E_0\otimes \phi^*(F)\otimes \O\left(\sum \lambda_i w_i\right)\right)^{\mu_2}\right) = e+d^2. 
        \end{equation*}
        The same argument as in Case 1 applies and implies that $2\deg(F)+|I|=0$ where $I$ is the subset of $\{1,...,dk\}$ such that $\lambda_i=1$ and $|I|$ is its cardinality. Hence, $\Upsilon^{-1}(C)$ is again isomorphic to $\bigsqcup_I\Pic^{-|I|/2}(C)$. Note that although we use the auxiliary line $H$ and the divisor $D_a$ in the argument, the result is independent of them.

        \end{proof}

        \begin{rmk}
        Note that the isomorphism $\Upsilon^{-1}(C) \cong \bigsqcup_I \Pic^{-|I|/2}(C)$ here is not canonical, as $E_0$ is only determined up to tensorization by line bundles. 
        \end{rmk}
        

        Suppose $d=1,2$. For $d<\deg(\Delta)$, if we call the line bundle $L_d: =O_{\P^2}(d)|_\Delta $ on $\Delta$, it is easy to see that $|O_{\P^2}(d)|\cong |L_d|$. Recall the group scheme $G|_{U_d}$ over $U_d$ defined in Section 2. 
        
        \begin{cor}\label{fiber is a torsor}
        With the same notation as above, for $d=1,2$, $\Upsilon ^{-1}(C)$ is a $G|_C$-torsor.
        \end{cor}

        \begin{proof}
        For $d=1,2$, the Picard group $\Pic^a(C)$ is trivial for any $a$. Let $\sum p_i= C\cap \Delta$ be the divisor corresponding to $C$ under $|\O_{\P^2}(d)|\cong |L_d|.$ We can denote a closed point of $G|_C$ by $\sum (\lambda_i,p_i)$ where $\lambda_i\in \Z/2\Z$ and $\sum \lambda_i= 0$. Since any $M\in \Upsilon^{-1}(C)$ can be written as $\psi_*\wh{M}$, the group $G|_C$ acts on $\Upsilon^{-1}(C)$ by 
        \begin{equation*}
            \left( \sum (\lambda_i,p_i)\right) \cdot M = \psi_* \left( \wh{M}\otimes \O\left(\sum\lambda_i\frac{p_i}{2}\right)\otimes h_C^{-\frac{1}{2}\sum \lambda_i} \right)
        \end{equation*}
        where $h_C= \psi^*\O_C(1)$. To see that $G|_C$ acts simply transitively, fix $E_0$ such that for $M\in \Upsilon^{-1}(C)$, $M\cong \psi_*(E_0\otimes \wh{L})$, then the action becomes
        \begin{equation*}
            \left( \sum (\lambda_i,p_i)\right) \cdot M = \psi_* \left( E_0\otimes \wh{L}\otimes \O\left(\sum\lambda_i\frac{p_i}{2}\right)\otimes h_C^{-\frac{1}{2}\sum \lambda_i} \right)
        \end{equation*}
        which is clearly simply transitive by the description of $\Upsilon^{-1}(C)$ in the proof of Theorem \ref{fibers as even card subsets}.
        \end{proof}

\section{Moduli spaces of $\mathcal{B}_0$-modules and special subvarieties of Prym varieties}
In this section, we will construct the rational map from the moduli space $\f M_{d,e}$ to the Prym variety $\Prym(\td{\Delta},\Delta)$. The key observation is that our $\m B_0$-modules are supported on plane curves $C$ which intersect the discriminant curve $\Delta$ in finitely many points. The $\m B_0$-modules restrict to a representation of the even Clifford algebra over each of these points. These representations then define a lift of the intersection $C\cap \Delta\subset \Delta$ to $\td{\Delta}$, which will be a point in the variety of divisors lying over the linear system $|\O_{\P^2}(C)|_{\Delta}|$, and maps to $\Prym(\td{\Delta},\Delta)$. So we begin by studying the representation theory for our purpose. 

\subsection{Representation theory of degenerate even Clifford algebras }
In this subsection, we will restrict our attention to the fiber of the sheaf of the even Clifford algebras $\m B_0$ over a fixed $p\in C\cap\Delta$, which is a $\C$-algebra, denoted by $A$. Note that all the fibers over the points in $C\cap\Delta$ are isomorphic as $\C$-algebra since the fiber $\m B_0|_p$ over a point $p\in C\cap \Delta$ is defined by a degenerate quadratic form of corank 1 and all quadratic forms of corank 1 are isomorphic over $\C$. Let $V$ be a vector space of dimension 3, and $q\in S^2V^*$ a quadratic form of rank 2. The even Clifford algebra is defined as a vector space $\C\oplus \wedge^2 V$ together with an algebra structure defined as follows. First, we can always find a basis $\{e_1,e_2,e_3\}$ of $V$ such that $q$ is represented as the matrix $\textrm{diag}(1,1,0)$ and we denote by $\{1,x:= ie_1\wedge e_2,y := ie_2\wedge e_3,z:=e_1\wedge e_3\}$ the basis of $\C\oplus \wedge^2 V$. The relations are given by 
\begin{equation}
    x^2 =1, y^2=z^2=0, xy=-z, xz=-y, xy = -yx, xz=-zx,  yz=zy=0.
\end{equation}

Since $A$ is a finite-dimensional associative algebra, we can understand it via quivers and path algebras. We refer the reader to \cite{assem_simson_skowronski06elements} for the basics of quivers and path algebras.

\begin{prop}\label{degenerate even Clifford = path algebra}
The algebra $A$ is isomorphic to the path algebra associated to the following quiver $Q$
\begin{equation}
\begin{tikzcd}
        +\arrow[r,bend left,"\a"]&-\arrow[l,bend left,"\b"]
\end{tikzcd}
\end{equation}
with relations $\a\b=\b\a= 0.$ 
\end{prop}
\begin{proof}
We begin by finding the idempotents i.e. elements $b$ in $A$ such that $b^2=b$. This is achieved by setting up the equations 
\begin{equation*}
    (a_0+a_1x+ a_2y +a_3z)^2 = (a_0+a_1x+ a_2y +a_3z) 
\end{equation*}
and solving the equations in $a_0,a_1,a_2,a_3$. It is easy to check that the idempotents are:
\begin{equation*}
    0,1, \frac{1}{2}(1\pm x) + a_2 y+a_3 z
\end{equation*}
for $a_2,a_3\in \C$ and that
\begin{equation*}
    e_+ := \frac{1}{2}(1+x), \quad e_-:=\frac{1}{2}(1-x)
\end{equation*}
is a complete set of primitive orthogonal idempotents of $A$. From the description of idempotents, it is clear that the only central idempotents are $0, 1$, so $A$ is connected. 

We also need to compute the radical of $A$. Observe that the ideal $I=(y,z)$ is clearly nilpotent, i.e. $I^2=0$ and $A/I \cong \C[x]/(x^2-1)\cong \C\oplus \C$. By \cite[Corollary 1.4(c)]{assem_simson_skowronski06elements}, this implies that $\rad(A)=I=(y,z).$ It also follows that $A$ is a basic algebra by \cite[Proposition 6.2(a)]{assem_simson_skowronski06elements}.

The arrows between $+ \to -$ of the associated quiver is described by 
\begin{equation*}
    e_-(\rad(A)/\rad(A)^2)e_+ =\frac{1}{2}(1-x)(y,z)\frac{1}{2}(1+x) = \C(y+z).
\end{equation*}
Similarly, the arrows between $-\to +$ is described by 
\begin{equation*}
    e_+(\rad(A)/(\rad(A)^2))e_- = \C(y-z)
\end{equation*}
and the arrows between $-\to -$ and $+\to +$
\begin{equation}
    e_-(\rad(A)/(\rad(A)^2))e_- = e_+(\rad(A)/(\rad(A)^2))e_+ = 0.
\end{equation}
Hence, the associated quiver $Q$ \cite[Definition 3.1]{assem_simson_skowronski06elements} of $A$ is given by 
\begin{equation*}
\begin{tikzcd}
        +\arrow[r,bend left,"\a"]&-\arrow[l,bend left,"\b"]
\end{tikzcd}
\end{equation*}
and we obtain a surjective map $    \C Q\to A$ from the path algebra $\C Q$ associated to the quiver $Q$ to $A$ by sending the generators 
\begin{equation}\label{isom with quiver algebra}
    e_+ \mapsto \frac{1}{2}(1+x) , \quad e_-\mapsto \frac{1}{2}(1-x), \quad \a \mapsto \frac{1}{2}(y+z) , \quad \b \mapsto \frac{1}{2}(y-z).
\end{equation}
It is easy to see that $\a\b=\b\a=0$, and since any other paths of higher length must contain a factor of $\a\b$ or $\b\a$, we see that the kernel of $kQ\to A$ must be $J= (\a\b,\b\a)$. Therefore, we have an isomorphism $\C Q/J\cong A$.

\end{proof}
\begin{rmk}
We can prove the isomorphism in Proposition \ref{degenerate even Clifford = path algebra} directly by checking that the map defined in (\ref{isom with quiver algebra}) is indeed an isomorphism of $\C$-algebras. The detail with the idempotents and the radical ideal in the proof above is just to display a more systematic approach. 
\end{rmk}

Since we are mainly interested in $\m B_0$-modules that are locally free of rank 2, the fiber of such module over $p\in C\cap\Delta$ is a representation of $A$ on $\C^2$. In light of the interpretation of $A$ as a path algebra, we can easily classify all the isomorphism classes of representations on $\C^2$. The isomorphism classes of representations of $A\cong \C Q/J$ on $\C^2$ are listed as follows: 
\begin{enumerate}
    \item \begin{equation*}\label{type 1 rep}
        \begin{tikzcd}
                \C \arrow[r,bend left,"1"]&\C\arrow[l,bend left,"0"] 
        \end{tikzcd}
    \end{equation*}
    
    \item \begin{equation*}\label{type 2 rep}
        \begin{tikzcd}
                \C \arrow[r,bend left,"0"]&\C\arrow[l,bend left,"1"] 
        \end{tikzcd}
    \end{equation*}
    
    \item \begin{equation*}
        \begin{tikzcd}
                \C \arrow[r,bend left,"0"]&\C\arrow[l,bend left,"0"] 
        \end{tikzcd}
    \end{equation*}
    
    \item \begin{equation*}
        \begin{tikzcd}
                \C^2 \arrow[r,bend left,"0"]&0\arrow[l,bend left,"0"] 
        \end{tikzcd}
    \end{equation*}
    
    \item \begin{equation*}
        \begin{tikzcd}
                0 \arrow[r,bend left,"0"]&\C^2\arrow[l,bend left,"0"] 
        \end{tikzcd}
    \end{equation*}
    
\end{enumerate}

\subsection{Construction}
    Recall the geometric set-up: we have a rank 3 bundle $F$ on $\P^2$ and an embedding of a line bundle $q: L \hookrightarrow S^2F^\vee$. This defines a conic bundle $Q\subset \P(F)$ as the zero locus of $q\in H^0(\P^2, S^2F^\vee \otimes L^\vee)\cong H^0(\P(F),\O_{\P(F)/\P^2}(2)\otimes (\pi')^* L^\vee)$ in $\P(F)$ where we denote $\pi': \P(F) \to \P^2$. The discriminant curve is assumed to be smooth and denoted by $\Delta$. As an $\O_{\P^2}$-module, the sheaf of even Clifford algebras $\m B_0$ on $\P^2$ is 
\begin{equation*}
    \m B_0\cong \O_{\P^2 }\oplus \wedge^2 F \otimes L. 
\end{equation*}

We will restrict our attention to $\m B_0$-modules supported on a degree $d$ smooth curve $C\subset \P^2$ which intersects $\Delta$ transversely. Given such a $\m B_0$-module $M$ on $C$, for each $p\in C\cap \Delta$ we consider the vector subspace 
\begin{equation*}
    K := \ker( \m B_0|_p\to \m End(M)|_p).
\end{equation*}
As we will see in Proposition \ref{vector subspaces} (1), $K$ is a vector subspace of $ \wedge^2 F|_p\otimes L|_p$. The natural isomorphisms $w:\wedge^2 F \xrightarrow{\sim}\det(F)\otimes F^\vee$ and $F\xrightarrow{\sim} (F^\vee)^\vee$ give rise to another vector space
\begin{equation*}
    K':= \ker( F|_p \xrightarrow{\sim} (F^\vee)^\vee|_p \xrightarrow{w_p^\vee\otimes \det(F)|_p\otimes L|_p }  K^\vee \otimes (L\otimes \det(F))|_p)
\end{equation*}
where $w_p: K\hookrightarrow (\wedge^2F\otimes L)|_p \to (\det(F)\otimes F^\vee \otimes L)|_p$ is the composition of the inclusion map and the isomorphism $w$ restricted to $p.$ Hence, $\P(K')$ is a linear subspace in $\P(F|_p)$. In the light of Proposition \ref{vector subspaces}, $K'$ is the two-dimensional vector space in $F|_p$ that corresponds to the line $K$ in $\wedge^2 F|_p$ (identified with ($\wedge^2 F\otimes L)|_p$). 
\begin{prop} \label{vector subspaces}

\end{prop}

\begin{enumerate}
    \item $K\subset \wedge^2 F|_p\otimes L|_p\subset \O|_p \oplus \wedge^2 F|_p\otimes L|_p$ ;
    \item $\dim (K) = 1$ and $\dim (K')=2$;
    \item The line $\P(K')\subset \P(F|_p)$ is one of the two irreducible components of the degenerate conic $Q|_p\subset \P(F|_p)$.
\end{enumerate}
\begin{proof}
First of all, we can always choose a basis $\{e_1,e_2,e_3\}$ of $F|_p$ and a trivialization $i: L|_p\cong \C $ so that $q|_p$ is represented by $\textrm{diag}(1,1,0)$. The trivialization $i$ induces an isomorphism of $\C$-algebras $\m B_0|_p\cong \O|_p\oplus \wedge^2 F|_p$  where the latter is generated by $1\in \O|_p$ and $\{x:= ie_1\wedge e_2,y := ie_2\wedge e_3,z:=e_1\wedge e_3\}\subset \wedge^2 F|_p$ with relation 
\begin{equation*}
    x^2 =1, y^2=z^2=0, xy=-z, xz=-y, xy = -yx, xz=-zx,  yz=0.
\end{equation*}

The irreducible components of $Q|_p\subset \P(F|_p)$ are given by the projectivization of the isotropic planes in $F|_p$ with respect to $q$. If we write a vector $v\in F|_p$ as $\sum_{i=1}^3a_ie_i$, then the two isotropic planes are given by the two equations 
\begin{equation}
    a_1+ia_2=0 , \quad a_1-ia_2=0
\end{equation}
which correspond to the lines in $\wedge^2 F|_p $
\begin{equation}
    \C \langle ie_2\wedge e_3+e_1\wedge e_3\rangle=\C\langle y+z\rangle , \quad \C \langle ie_2\wedge e_3- e_1\wedge e_3 \rangle =\C\langle y-z\rangle.
\end{equation}
To prove all the claims, it suffices to show that $K\subset \m B_0|_p$ corresponds to one of the these lines in the subspace $\wedge^2 F|_p\subset \O|_p\oplus \wedge^2 F|_p$. Indeed, then $K'$ will correspond to one of the isotropic planes. 

Recall that with the choice of basis $\{1,x,y,z\}$ of $\O|_p\oplus \wedge^2 F|_p$, we have an isomorphism $\m \C  Q/J \xrightarrow{\sim}\O|_p\oplus \wedge^2 F|_p$:
\begin{equation*}
    e_+ \mapsto \frac{1}{2}(1+x) , \quad e_-\mapsto \frac{1}{2}(1-x), \quad \a \mapsto \frac{1}{2}(y+z) , \quad \b \mapsto \frac{1}{2}(y-z).
\end{equation*}
Then the kernel of $\m B_0|_p\to \m End (M)|_p$ can be computed by the composition
\begin{equation*}
\C Q/J\xrightarrow{\sim}\O|_p\oplus \wedge^2 F|_p\xrightarrow{\sim }\m B_0|_p\to \m End(M)|_p
\end{equation*}
which is a representation of the path algebra $\C Q/J$. As we will see in Proposition \ref{type 1-2 argument}, the isomorphism classes of the representation of $\C Q/J$ on $\C^2$ in this case must be either type (1) and (2). Taking this as granted for a moment, we have: 

\begin{enumerate}
    \item For type (1), the kernel of $\C Q/J\xrightarrow{\sim }\m B_0|_p\to \m End(M)|_p$ is $\C\langle\beta\rangle $ which corresponds to $K=\C\langle y-z\rangle\subset \O|_p\oplus \wedge^2 F|_p. $
    \item For type (2), the kernel of $\C Q/J\xrightarrow{\sim }\m B_0|_p\to \m End(M)|_p$ is $\C\langle\a\rangle $ which corresponds to $K=\C\langle y+z\rangle \subset \O|_p\oplus \wedge^2 F|_p$.
\end{enumerate}
All the claims follow immediately. 
\end{proof}

\begin{rmk}
The trivialization $i:L|_p\cong \C$ does not cause any ambiguity in the identifications; as we are only interested in identification of vector subspaces, other trivializations will only differ by a scalar multiplication. 
\end{rmk}

\begin{prop}\label{type 1-2 argument}
    The representation of $\m B_0|_p$ obtained from a $\m B_0$-module $M$ as the fiber $M|_p$ over $p\in C\cap \Delta$ must have isomorphism class of either type (\ref{type 1 rep}) or type (\ref{type 2 rep}).
    \end{prop}

    \begin{proof}
        Fix $p\in C\cap \Delta$. Let $n=3,4,5$ and $M_n$ be a $\m B_0$-modules such that its fiber over $p$ is a $\m B_0$-representation of type $n$ isomorphism class. We can choose a local parameter $t\in \O_{C,p}$ as $\O_{C,p}$ is a discrete valuation ring. 
        
        Then $M_n$ induces the homomorphisms over the local ring $\O_{C,p}$ and over the residue field $\kappa(p)$ (i.e. fiber)
      \begin{equation*}
           \rho_n: \m B_0\otimes \O_{C,p} \to \m End (M_n) \otimes \O_{C,p}, \quad \rho_n^0: \m B_0\otimes \kappa(p) = \m B_0|_p \to \m End (M_n) \otimes \kappa(p)= \m End(M_n)|_p.
      \end{equation*}
        Again, we can always choose a basis $\{e_1,e_2,e_3\}$ of $F|_p$ and a trivialization $i: L|_p\cong \C $ so that $q|_p$ is represented by $\textrm{diag}(1,1,0)$. The trivialization $i$ induces an isomorphism of $\C$-algebras $\m B_0|_p\cong \O|_p\oplus \wedge^2 F|_p$  where the latter is generated by $1\in \O|_p$ and $\{x:= ie_1\wedge e_2,y := ie_2\wedge e_3,z:=e_1\wedge e_3\}\subset \wedge^2 F|_p$ with relations 
\begin{equation*}
    x^2 =1, y^2=z^2=0, xy=-z, xz=-y, xy = -yx, xz=-zx,  yz=0.
\end{equation*}
        
        Recall that the algebra $\O|_p\oplus \wedge^2 F|_p$ is isomorphic to $\C Q/J$ generated by $e_+,e_-,\a,\b.$ If we call the isomorphism $j:\C Q/J\xrightarrow{\sim }\O|_p\oplus \wedge^2 F|_p \xrightarrow{\sim} \m B_0|_p $, then clearly $\rho^{0}_n(j(\a)) = \rho^{0}_n(j(\b))= 0$. It follows that $\rho^{0}_n(y) = \rho^{0}_n ( j(\a+\b)) = 0$ and similarly $\rho^{0}_n(z) = 0.$ So that means $\rho_n(y) = tP_y $ and $\rho_n(z) = tP_z$ for some $P_y,P_z\in \m End(M_n)\otimes \O_{C,p}$.

        As $F$ is locally free of rank 3, by Nakayama lemma, we can lift the basis $\{e_1,e_2,e_3\}$ of $F|_p$ to a basis (also denoted as $\{e_1,e_2,e_3\}$ by abuse of notation) of $F\otimes \O_{C,p}$ over $\O_{C,p}$. The quadratic form $q$ is represented by the matrix 
        \begin{equation}
           (f_{ij})=  \begin{pmatrix}
                    f_{11}&f_{12}&f_{13}\\
                    f_{21}&f_{22}&f_{23}\\
                    f_{31}&f_{32}&f_{33}
            \end{pmatrix}
        \end{equation}
        where $f_{ij}$ are elements in $\O_{C,p}$. By the choice of basis $\{e_1,e_2,e_2\}$, we have $f_{11}|_{p} \neq 0$,  $f_{22}|_{p} \neq 0$ and $f_{ij}|_{p} = 0$ for $(i,j)\neq (1,1),(2,2)$. It follows that $f_{ij} = tf_{ij}'$ for $(i,j)\neq (1,1),(2,2)$ and $f_{ij}'\in \O_{C,p}$. Then in $\m B_0\otimes \O_{C, p}$ we have
        \begin{align*}
            yz &= (ie_2e_3)(e_1e_3) \\
            &= if_{31}e_2e_3 - ie_2e_1e_3e_3 \\
            &= if_{31}e_2e_3 -if_{33}e_2e_1\\
            &= f_{31}y  - if_{33}f_{21} + f_{33}x
        \end{align*}
        (here we omit the $"\wedge"$ between the $(e_i)'s$)
        so it follows that 
\begin{equation*}
       \rho_n(yz) = t^2f_{31}'P_y - it^2f_{33}'f_{21}' + tf_{33}'\rho_n(x).
\end{equation*}
        Since
        \begin{equation*}
            \rho_n(yz) = \rho_n(y)\rho_n(z) = t^2P_yP_z
        \end{equation*}
        by equating the two expression, we get
\begin{equation*}
        t^2P_yP_z = t^2f_{31}'P_y - it^2f_{33}'f_{21}' + tf_{33}'\rho_n(x) \implies tP_yP_z = tf_{31}'P_y - itf_{33}'f_{21}' + f_{33}'\rho_n(x).
\end{equation*}
        Note that $f'_{33}$ is invertible in $\O_{C,p}$ because otherwise $\det(f_{ij})$ will have zeros of order 2 with respect to $t$, which is not allowed since we assume that $C$ intersects $\Delta$ transversally. Hence, we can write $\rho_n(x) = tP_x$ for some $P_x\in \m End(M_n)\otimes \O_{C,p}$. In particular, we must have $\rho_n^0(x^2) =0.$

        On the other hand, we have
        \begin{align*}
            x^2 &= (ie_1e_2)(ie_1e_2)\\
            &=-f_{21}e_1e_2 +e_1e_1e_2e_2\\
            &= if_{21}x  + f_{11}f_{22}
        \end{align*}
        (again we omit the $"\wedge"$ between the $(e_i)'s$)
        and so $\rho_n^0(x^2)  =( f_{11}f_{22})|_p\neq 0$ as  $f_{21}|_p=0$, $f_{11}|_p\neq 0$ and $f_{22}|_p\neq 0.$ Hence, a contradiction.  
        
    \end{proof}
    
    Let $U_d\subset |\O_{\P^2}(d)|$ be the subset of smooth curves of degree $d$ which intersect $\Delta$ transversally. For $d<\deg(\Delta)=k$, if we call the line bundle $L_d: =O_{\P^2}(d)|_\Delta $ on $\Delta$, it is easy to see that $|O_{\P^2}(d)|\cong |L_d|$. Hence, we can consider the  variety of divisors $W_d$ lying over $|L_d|$ and its two components as $W^i_d$ for $i=0,1$.

    For each $\m B_0$-module $M\in \f M_{d,e}$ with support on $C\in U_d,$ let $j:C\cap \Delta \hookrightarrow C$ be the inclusion, by Proposition \ref{vector subspaces}, the assignment 
    \begin{equation*}
    M\mapsto \m K:= \ker \left(j^*\m B_0\to  j^*\m End(M)\right) 
    \end{equation*}
     is argued to be contained in $j^*(\wedge^2 F)$ and it defines exactly a point in $W_d.$ 
    
    This construction also works in families. Let $T$ be a scheme and $\m M_T$ be a flat family of $\m B_0$-modules on $\P^2$ with supports on curves of $U_d$ and Chern character $(0,2d,e)$ i.e. $\m M_T$  is a $p^*_1\m B_0$-module on $\P^2\times T$ flat over $T$ with $\textrm{Fit}(\m M_t)\in U_d$ and $\textrm{ch}(\m M_t)=(0,2d,e)$ for all $t:\Spec(\C)\to T$ where $\m M_t= t^*\m M_T$ and $p_1:\P^2\times T\to \P^2$ is the projection. Then we get a map $T\to \f M_{d,e}|_{U_d}\to |U_d|\subset  \Delta^{(dk)}$. We can restrict the family of $\m B_0$-modules to $ \Delta\times T\subset \P^2\times T.$ Consider the universal divisor 
     \begin{equation*}
         \begin{tikzcd}
         &\m D \subset \Delta\times \Delta^{(dk)}\arrow[ld]\arrow[rd]&\\
         \Delta&&\Delta^{(dk)}
         \end{tikzcd}
     \end{equation*}
     By pulling back $\m D $ along the map $\Delta\times T \to \Delta\times \Delta^{(dk)}$, we get another divisor $\m D_T\subset \Delta\times T$ and denote the inclusion by $i_T: \m D_T\hookrightarrow\Delta\times T\hookrightarrow \P^2\times T$. 
     
     We will write $F_T:= i_T^*p_1^*F$ and $L_T:= i_T^*p_1^*L$. The sheaf 
     \begin{equation*}
         \m K_T:= \ker\left( i_T^* p_1^* \m B_0 \to i_T^*\m End (\m M_T)  \right)
     \end{equation*} 
     has constant fiber dimension one and is contained in the rank 3 vector bundle $i_T^*p_1^*\left(\wedge^2F_T\otimes L_T\right)$ on $\m D_T$ by Proposition \ref{vector subspaces}, where $F_T=p_1^*F$ and $L_T=p_1^*L$. Again, since there are the natural isomorphisms $w:\wedge^2 F_T\xrightarrow{\sim} \det(F_T)\otimes F_T^\vee$ and $F_T\xrightarrow{\sim}(F_T^\vee)^\vee$, we can define
     \begin{equation*}
         \m K_T':= \ker( F_T \xrightarrow{\sim} (F_T^\vee)^\vee \xrightarrow{w_T^\vee\otimes \det(F_T)\otimes L_T }  \m K_T^\vee \otimes L_T\otimes \det(F_T))
     \end{equation*}
     where  $w_T:\m K_T \hookrightarrow \wedge^2F_T\otimes L_T \to \det(F_T)\otimes F_T^\vee \otimes L_T$ is the composition. As we checked in Proposition \ref{vector subspaces} that each fiber of the projectivization $\P(\m K'_T)\subset \P(F_T)$ is a component of the fiber of a degenerate conic in the conic bundle $Q\to \P^2$, so we have $\P(\m K'_T)\subset i_T^*p_1^*Q\subset \P(F_T)$. Since $\td{\Delta}\to \Delta$ is the curve parametrizing the irreducible components of $Q|_\Delta \to \Delta$, it follows that $\P(\m K'_T)$ over $\m D_T$ defines a divisor $\td{\m D}_T\subset \td{\Delta}\times T$ that maps to $\m D_T$ via $\td{\Delta}\times T\to \Delta\times T$. The divisor $\td{\m D}_T$ is a $T$-family of degree $dk$ divisors on $\td{\Delta}$, so it defines a map $T\to \td{\Delta}^{(dk)}$ which factors through $W_d|_{U_d}$ since $\m D_T$ is induced from a map $T\to U_d$. It is easy to check that the assignment from $\m M_T$ to $T\to W_d$ is functorial, hence we obtain a morphism over $U_d$: 
    \begin{equation}
    \begin{tikzcd}
    \f M_{d,e}|_{U_d} \arrow[rr,"\Phi"]\arrow[rd]&& W_d|_{U_d}\arrow[ld]\\
    &U_d&
    \end{tikzcd}
    \end{equation}

    \begin{prop} \label{morphism of torsors}
    The morphism $\Phi|_C: \f M_{d,e}|_{C} = \Upsilon^{-1}(C) \to W_d|_{C} $ over $C\in U_d$ is $G|_{C}$-equivariant. 
    \end{prop}
    \begin{proof}
     Let $\sum (\lambda_i,p_i)\in G|_C$. Recall the notations from Theorem \ref{fibers as even card subsets} that $\wh{C}= C_{C\cap \Delta, 2}$ is the 2nd-root stack and $\psi:\wh{C}\to C$ is the projection morphism. By Theorem \ref{fibers as even card subsets},  we can write $M\in \f M_{d,e}|_{C}$ as $M= \psi_*(E_0\otimes \wh{L})$ for a $\wh{L}$ line bundle on $\wh{C}$ by choosing a rank 2 bundle $E_0$ (recall that $E_0$ is determined up to a line bundle) on $\wh{C}$. We need to show that 
    \begin{equation}
       \Phi\left( \sum\left(\lambda_i,p_i\right) \cdot M \right) =\Phi\left(\psi_*\left( E_0\otimes\wh{ L}\otimes\O\left(\sum_i \frac{\lambda_i}{2}p_i\right)\otimes h_C^{-\frac{1}{2}\sum \lambda_i} \right)\right)  = \left(\sum \lambda_ip_i\right)\cdot \Phi (M)
    \end{equation}
     Since $\Phi(M)$ is determined at each point in $C\cap \Delta$, it suffices to check the equivariance property over a point $p\in C\cap \Delta$. As we checked that $\ker\left(\m B_0|_p\to \m End\left.\left(\psi_*\left(E_0\otimes \wh{L}\right)\right)\right|_p\right)$ always determines one of the two preimages of $p\in \Delta$ in the double cover $\td{\Delta},$ to prove the proposition it suffices to show that 
    \begin{equation*}
        \ker\left(\m B_0|_p\to \m End\left.\left(\psi_*\left(E_0\otimes \wh{L}\right)\right)\right|_p\right)\neq \ker \left(\m B_0|_p\to \m End\left.\left(\psi_*\left(E_0\otimes \wh{L}\otimes \O\left(\frac{p}{2}\right)\right)\right)\right|_p\right)
    \end{equation*}
    or equivalently, 
    \begin{equation}\label{equivariance property}
        \ker\left(\m B_0|_p\to \m End\left.\left(\psi_*\left(E_0\otimes \wh{L}\right)\right)\right|_p\right)\neq \ker \left(\m B_0|_p\to \m End\left.\left(\psi_*\left(E_0\otimes \wh{L}\otimes \O\left(-\frac{p}{2}\right)\right)\right)\right|_p\right).
    \end{equation}
    In fact, we can simplify further by assuming $\wh{L}= \O_{\wh{C}}$.
    
    The $\m B_0$-module structure on $\psi_*( E_0\otimes \wh{L})$ can be described concretely by the composition of the isomorphism $\m B_0\cong \psi_*\m End(E_0)\cong \psi_*\m End(E_0\otimes \wh{L})$ and the natural morphism
    \begin{equation*}
           \a: \psi_*\m End(E_0\otimes \wh{L})\to \m End(\psi_*(E_0\otimes \wh{L})).
    \end{equation*}
    In particular, we can define 
    \begin{align*}
        \a^0&:\psi_*\m End(E_0)\xrightarrow{\sim} \psi_*\m End\left(E_0\otimes \O\left(-\frac{p}{2}\right)\right)\to \m End\left(\psi_*\left(E_0\otimes \O\left(-\frac{p}{2}\right)\right)\right)\\
        \a^1&: \psi_*\m End(E_0)\to \m End\left(\psi_*(E_0)\right)
    \end{align*}
    Hence, to check that (\ref{equivariance property}) holds, it is equivalent to show that $\ker(\a^0|_p)\neq \ker(\a^1|_p)$. 
    
    To check this, we proceed as in Example \ref{root stack over affine scheme} and Proposition \ref{characters of E0} and work in an affine neighborhood $Z= \Spec(R)$ of $p$ and the double cover $\td{Z}= \Spec(R')$ where $R':=R[t]/(t^2-s))$ and $div(s)=p$. So that the root stack restricted over $Z$ is simply $\wh{Z}= [ \Spec(R[t]/(t^2-s))/\mu_2]$. Recall the notations that a $\Z/2\Z$-graded $R'$-module is written as $A=A_0\oplus A_1$ with $A_i$ being the graded pieces. We can further reduce to the localization of $R$ at $p$, we will again write the local ring as $R$ and its unique maximal ideal $\f m$ which contains $s.$

    As argued in Proposition \ref{characters of E0}, $E_0$ is a $\Z/2\Z$-graded $R'$-module $N= N_0\oplus N_1$ and  we can choose $e_1\in N_0 $ and $e_2\in N_1$ such that $N\cong R'e_1\oplus R'e_2$. In terms of the $R'$-basis $\{e_1,e_2\}$, $\psi_*(\m End(E_0))\cong \left(\m End(E_0)\right)_0$ consists of homogeneous $R$-module homomorphisms $\delta$ of degree 0:
     \begin{equation}\label{endomorphism 1}
         \begin{aligned}
              e_1\mapsto u_0e_1 + u_1e_2\\
         e_2\mapsto v_1e_1+ v_0e_2
         \end{aligned}
     \end{equation}
     where $u_i,v_i\in (R')_i. $ Then we can write $u_1 = t\td{u}_1$ and $v_1 = t\td{v}_1$ where $\td{u}_1,\td{v}_1\in (R')_0=R.$      
        
    As before, the module $\psi_*E_0$ is freely generated by $\{f_1=e_1,f_2= te_2\}$ as $R$-module. Suppose that $\delta\in \psi_*End(E_0)$ is of the form (\ref{endomorphism 1}), then its image in $\m End(\psi_*E_0)$ under $\a$ will be a map of the form
    \begin{align*}
        f_1 = e_1&\mapsto u_0e_1 + \td{u}_1(te_2)  = u_0f_1+ \td{u}_1f_2\\
        f_2 = te_2&\mapsto \td{v}_1t(te_1) + v_0(te_2)  = s\td{v}_1f_1 + v_0f_2
    \end{align*}
    If we choose the generators of $\psi_*\m End(E_0)$ to be the following $R$-valued matrices (with respect to the basis $\{e_1,e_2\})$
    \begin{equation}
        I= \begin{pmatrix}
            1&0\\0&1
        \end{pmatrix}, \quad
        a:= \begin{pmatrix}
            -1&0\\0&1
        \end{pmatrix}, \quad 
        b:= \begin{pmatrix}
            0&t\\t&0
        \end{pmatrix}, \quad 
        c:=\begin{pmatrix}
            0&t\\-t&0
        \end{pmatrix}
    \end{equation}
    their images in $\m End(\psi_*E_0)$ are the corresponding $R$-matrices (with respect to the basis $\{f_1,f_2\}$):
    \begin{equation}
        \begin{pmatrix}
            1&0\\0&1
        \end{pmatrix}, \quad
        \begin{pmatrix}
            -1&0\\0&1
        \end{pmatrix}, \quad 
        \begin{pmatrix}
            0&s\\1&0
        \end{pmatrix}, \quad 
        \begin{pmatrix}
            0&s\\-1&0
        \end{pmatrix}
    \end{equation}
     
    As the homomorphism $\O\left(-\frac{p}{2}\right)\to \O_{\wh{C}}$ is represented as the inclusion $R't\hookrightarrow R'$ of $R'$-modules, the homomorphism $E_0\otimes \O\left(-\frac{p}{2}\right)\to E_0$ corresponds to taking the $R'$-module homomorphism
     \begin{equation}\label{injection}
         R'(te_1)\oplus R'(te_2)\to R'(e_1)\oplus R'(e_2).
     \end{equation}
     Note that $\psi_*\left(E_0\otimes \O\left(-\frac{p}{2}\right)\right)$ corresponds to the $R$-module $R(sf_1)\oplus Rf_2$ as the $\mu_2$-invariant submodule of $R'(te_1)\oplus R' (te_2)$.
     Pushing the homomorphism (\ref{injection}) forward $\psi_*\left(E_0\otimes \O\left(-\frac{p}{2}\right)\right)\to \psi_*E_0$ corresponds to taking the $\mu_2$-invariant part $R(sf_1)\oplus Rf_2 \to Rf_1\oplus Rf_2$, which is represented by the $R$-valued matrix
     \begin{equation*}
         \begin{pmatrix}
                s&0\\0&1
         \end{pmatrix}
     \end{equation*}
     with respect to the bases $\{sf_1,f_2\}$ of $\psi_*E_0\otimes \O(-\frac{p}{2})$ and $\{f_1,f_2\}$ of $\psi_*E_0$.

     For any $\delta \in \psi_*(\m End(E_0))$, there are $\O_Z$-module homomorphisms $\a^0(\delta) :\psi_*(E_0\otimes \O(-\frac{p}{2}))\to \psi_*(E_0\otimes \O(-\frac{p}{2}))$ and $\a^1(\delta):\psi_*E_0\to \psi_*E_0$. They form a commutative diagram by the definition of a $\psi_*\m End(E_0)$-module homomorphism
     \begin{equation*}
         \begin{tikzcd}
         \psi_*(E_0\otimes \O(-\frac{p}{2}))\arrow[r,"\a^0(\delta)"]\arrow[d]& \psi_*(E_0\otimes \O(-\frac{p}{2}))\arrow[d]\\
         \psi_*E_0\arrow[r,"\a^1(\delta)"]&\psi_*E_0
         \end{tikzcd}
     \end{equation*}
     
     In terms of the bases $\{f_1,f_2\},\{sf_1,f_2\}$ of $\psi_*E_0$ and $\psi_*(E_0\otimes \O(-\frac{p}{2}))$, the morphisms above can be written as 
        \begin{equation*}
         \begin{tikzcd}[ampersand replacement=\&]
         R\oplus R\arrow{rr}{\a^0(\delta)}\arrow{d}{\begin{pmatrix}
             s&0\\0&1
         \end{pmatrix}}\&\& R\oplus R\arrow{d}{\begin{pmatrix}
             s&0\\0&1
         \end{pmatrix}}\\
         R\oplus R\arrow[rr,"\a^1(\delta)"]\&\&R\oplus R
         \end{tikzcd}
     \end{equation*}
     Since $\a^1(b)= \begin{pmatrix}
         0&s\\1&0
     \end{pmatrix}$, it is easy to check that $\a^0(b)$ must be $\begin{pmatrix}
         0&1\\s&0
     \end{pmatrix}$. Similarly, we have the following $\a^0(\delta)$ when $\a^1(\delta)$ is the other generator:
     \begin{align*}
         &\a^1(I) = \begin{pmatrix}
             1&0\\0&1
         \end{pmatrix}\implies \a^0(I)= \begin{pmatrix}
             1&0\\0&1
         \end{pmatrix} \\
         &\a^1(a) = \begin{pmatrix}
             -1&0\\0&1
         \end{pmatrix}\implies \a^0(a)= \begin{pmatrix}
             -1&0\\0&1
         \end{pmatrix}  \\
         &\a^1(c) = \begin{pmatrix}
             0&s\\-1&0
         \end{pmatrix}\implies \a^0(c)= \begin{pmatrix}
             0&1\\-s&0
         \end{pmatrix} 
     \end{align*}
     Finally, when $s=0$ i.e. over $p$, we have 
     \begin{enumerate}
         \item the kernel of $\a^1|_p$ is spanned by $(b+c)|_p$,
         \item the kernel of $\a^0|_p$ is spanned by $(b-c)|_p $.
     \end{enumerate}
     Hence, $\ker(\a^0|_p)\neq \ker(\a^1|_p)$ and we are done.

    \end{proof}
    
    \begin{thm} \label{parity and birational type }
    Let $d=1,2$ and $e\in \Z$. 
    \begin{enumerate}
        \item The moduli space $\f M_{d,e}$ is birational to one of the two connected components $W_d^i$ of $W_d.$
        \item If $\f M_{d,e}$ is birational to $W^i_d$, then $\f M_{d,e+1}$ is birational to $W_d^{1-i}.$ In particular, the birational type of $\f M_{d,e}$  only depends on $d$ and $(e$ mod $2)$. 
    \end{enumerate}
    \end{thm}
    \begin{proof}
        In this proof, we will write $\Phi$ as $\Phi_{d,e}: \f M_{d,e}|_{U_d}\to W_d|_{U_d} $ to indicate $d$ and $e$ explicitly. 

        Since the morphism $\Upsilon|_{U_d}: \f M_{d,e}|_{U_d}\to U_d$ is quasi-finite (Theorem \ref{fibers as even card subsets}), there exists an open dense subset $V\subset U_d$ such that the restriction $\Upsilon|_{V}: \f M_{d,e}|_V\to V$ is finite.  Note that Corollary \ref{morphism of torsors} implies that each fiber $\Upsilon^{-1}(C)$ for $C\in V$ must be contained in a connected component of $W_d|_V$. It follows that $\Phi_{d,e}(\f M_{d,e} |_V)$ is contained in one of the connected components $W^i_d|_V$ of $W_d|_V$. Indeed, if this is not the case, the finite morphism $\Upsilon|_V$ would send the disjoint nonempty closed subsets $\Phi_{d,e}^{-1}(W^1_d|_V)$ and $\Phi_{d,e}^{-1}(W^2_d|_V)$ to disjoint nonempty closed subsets in $V$, contradicting the irreducibility of $V$. Then, the combination of Proposition \ref{pseudo torsor}, Corollary \ref{fiber is a torsor}, Proposition \ref{morphism of torsors} shows that the morphism $\Phi_{d,e}|_{\Upsilon^{-1}(V)}$ is a bijection on closed points from $\f M_{d,e}|_{V}$ to $W^i_d|_V$. As $W_d^i|_{V}$ is smooth and hence normal, the morphism $\Phi_{d,e}|_{\Upsilon^{-1}(V)}$ is an isomorphism and this proves part (1).

        Now, for any $e\in \Z$, suppose $M = \psi_*\wh{M}\in \f M_{d,e}$ and $\Phi_{d,e}(M)= x_1+...+x_{dk} \in W^i_d$ where $x_j\in \td{\Delta}$ and $k=\deg(\Delta)$. The computation in Theorem \ref{fibers as even card subsets} shows that $\ch\left(\psi_*\left(\wh{M}\otimes \O\left(\frac{p_j}{2}\right)\right)\right) = (0,2d,e+1)$. By the proof of Proposition \ref{morphism of torsors}, we see that $\Phi_{d,e+1}\left(\psi_*\left(\wh{M}\otimes \O\left(\frac{p_j}{2}\right)\right)\right) = x_1 + ...+\sigma(x_j)+ ... +x_{dk}\in W^{1-i}_d$. Hence, it follows that $\f M_{d,e+1}$ is birational to $W^{1-i}_d$ by part (1).

    \end{proof}

    \begin{rmk}
        The need for the assumption $d=1,2$ in Theorem \ref{parity and birational type } can already be seen by comparing the fiber dimension of $\Upsilon|_{U_d}: \f{M}_{d,e}|_{U_d}\to  U_d$ and $W_d|_{U_d}\to U_d$: the dimension of the fibers of the latter is 0, while the dimension of the fibers of the former is $\dim(\Pic(C)) = g(C)$ (by Theorem \ref{fibers as even card subsets}) for $C\in U_d$, which is positive for $d\geq 3.$
    \end{rmk}
\section{Cubic threefolds}

We will apply the construction of the rational map $\Phi: \f M_{d,e} \dashrightarrow W_d$ for the conic bundles obtained by blowing up smooth cubic threefolds along a line. As a consequence, this yields an explicit correspondence between instanton bundles on cubic threefolds and twisted Higgs bundles on the discriminant curve.

Let $Y\subset \P^4$ be a cubic threefold and $l_0\subset Y$ a general line. The blow-up $\sigma: \td{Y}:= Bl_{l_0}Y \to Y$ of $Y$ along $l_0$ is known to be a conic bundle $\pi: \td{Y}\to \P^2$. In this case, the rank 3 vector bundle is $F = \O_{\P^2}^{\oplus 2} \oplus \O_{\P^2}(-1)$ and the line bundle is $L = \O_{\P^2}(-1)$. The discriminant curve $\Delta$ of the conic bundle $\pi:\td{Y}\to \P^2$ is a degree 5 curve and its \'{e}tale double cover is denoted by $\td{\Delta}\to \Delta.$ Then we can consider the variety of divisors $W_2\subset \td{\Delta}^{(10)} $ lying over the linear system $|\O_{\P^2}(2)|_\Delta|$, its two components $W_2= W_2^0\cup W_2^1$ and the associated sheaf of even Clifford algebras $\m B_0$, and the moduli space $\f M_{d,e}$ as considered in previous sections.  Note that in the case of $Y$, $\O_{\P^2}(2)|_\Delta\cong K_\Delta$, so we can apply Example \ref{linear system of canonical divisors}. Recall that in Example \ref{linear system of canonical divisors} the Abel-Jacobi map $\td{\a}:\td{\Delta}^{(10)}\to J^{10}\td{\Delta}$ induces the morphism $\td{\a}|_{W_2^1}:W^1_2\to \Pr^1 $ that maps birationally to the abelian variety $\Pr^1$ and the morphism $\td{\a}|_{W_2^0}: W^0_2\to \Pr^0$ which is a generically $\P^1$-bundle over the theta divisor.

\begin{prop} \label{B_0 modules moduli = Prym}
Let $e\in \Z$ be even. The image of $\Phi: \f M_{2,e}\dashrightarrow W_2$ is contained in the component $W_2^1$. In particular, $\f M_{2,e}$ is birational to the Prym variety $\Prym(\td{\Delta},\Delta)$.
\end{prop}

\begin{proof}
Recall that Theorem \ref{parity and birational type } says that $\Phi$ maps $\f M_{2,-4}$ birationally to one of the connected components $W^i_2$ of $W_2$, it suffices to show that $\f M_{2,-4}$ cannot be birational to $W^0_2$. 
 By the work of \cite{Lahoz15ACM} (see Theorem \ref{instanton moduli = intermediate  J} and Theorem \ref{instanton moduli = B_0 modules moduli} in the next subsection), it is known that $\f M_{2,-4}$ is birational to another abelian variety, namely the intermediate Jacobian of the cubic threefold $Y$. In particular, the component of $W_2$ that is birational to $\f M_{2,-4}$ is birational to an abelian variety.

But recall from Example \ref{linear system of canonical divisors} (2) that $W^0_2$ is generically a $\P^1$-bundle, which cannot happen for a variety birational to an abelian variety. Hence, the image of $\Phi$ must be contained in $W^1_2.$ It follows immediately that the composition $\f M_{2,-4} \dashrightarrow W^1_2 \xrightarrow{\td{\a}|_{W^1_2}} \Pr^1$ is a birational map. By Theorem \ref{parity and birational type }, the same holds for $\f M_{2,e}$ when $e$ is even. 

\end{proof}

\begin{prop}
Let $e\in \Z$ be even. The image of $\Phi: \f M_{2,e+1}\dashrightarrow W_2$ is contained in the component $W^0_2$ and its image in $\Pr^0\cong Prym(\td{\Delta},\Delta)$ is an open subset of the theta divisor of the Prym variety. 
\end{prop}
\begin{proof}
This follows immediately from Theorem \ref{parity and birational type }, Proposition \ref{B_0 modules moduli = Prym}, and Example \ref{linear system of canonical divisors}. 
\end{proof}

\subsection{Instanton bundles on cubic threefolds and twisted Higgs bundles}

A rank 2 vector bundle $E$ on $Y$ is called an instanton bundle of minimal charge if $E$ is Gieseker semistable and $c_1(E)=0,c_2(E)=2$ and $c_3(E)=0.$ We will simply call it an instanton bundle for the rest of this section.

It is known that (see e.g. \cite{Druel}) there exist the moduli space of stable instanton bundles $\f M_Y$ and its compactification by the moduli space of semistable instanton sheaves $\overline{\f M}_Y$. Now, the intermediate Jacobian $J(Y)$ of a cubic threefold $Y$ has birationally a modular interpretation as the moduli space $\f M_Y$ of instanton bundles, via Serre's construction by the works of Markushevich, Tikhomirov, Iliev and Druel: 
    \begin{thm}[\cite{markushevich1998}\cite{Iliev2000}\cite{Druel}\cite{beauville02cubic}]\label{instanton moduli = intermediate  J}
    The compactification of $\f M_Y$ by the moduli space $\overline{\f M}_Y$ of rank 2 semistable sheaves with $c_1=0,c_2=2,c_3=0$ is isomorphic to the blow-up of $J(Y)$ along a translate of $-F(Y)$. Moreover, it induces an open immersion of $\f M_Y$ into $J(Y). $ 
    \end{thm}

We recall a theorem in \cite{Lahoz15ACM} relating instanton bundles and $\m B_0$-modules. Recall that we can embed the Fano surface of lines $F(Y)$ in $J(Y)$ as $F(Y)\hookrightarrow Alb(F(Y))\xrightarrow{\sim}J(Y)$ by picking $l_0$ as the base point. We denote by $\overline{F(Y)}$ the strict transform of $F(Y)$ under the blow-up in Theorem \ref{instanton moduli = intermediate  J}.
\begin{thm}[\cite{Lahoz15ACM}]\label{instanton moduli = B_0 modules moduli}
The moduli space $\f M_{2,-4}$ is isomorphic to the blow-up of $\overline{\f M}_Y$ along the strict transform $\overline{F(Y)}$ of $F(Y). $ In particular, $\f M_Y$ is birational to $\f M_{2,-4}.$
\end{thm}
For a stable instanton bundle $E\in \overline{\f M}_Y$, the image of $E$ in $ \f M_{2,-4}$ under the birational map in Theorem \ref{instanton moduli = B_0 modules moduli} is induced by a functor $\Xi:\m Ku(Y) \hookrightarrow \D^b(\P^{2},\m B_0) $ which can be described explicitly as follows. First, we define the functor 
\begin{equation*}
    \Psi :\textrm{D}^b(\td{Y}) \to \textrm{D}^b(\P^2,\m B_0), \quad E\mapsto \pi_*((E) \otimes_{\O_{\td{Y}}} \m E \otimes_{\O_{\td{Y}}} \det F^\vee[1])
\end{equation*}
where $\m E$ is a rank 2 vector bundle with a natural structure of flat left $\pi^*\m B_0$-module. For details of the definition, we refer to \cite{kuznetsov2008derived}. Then $\Xi(E) = \Psi(\sigma^*(E))$ which makes sense as it can be checked that instanton bundles on $Y$ are naturally objects in $\m Ku(Y)$ \cite[Section 3.2]{Lahoz15ACM}. While $\Xi(E)$ is a priori a complex, it turns out that $\Xi(E)$ is concentrated in only one degree \cite[Lemma 3.9]{Lahoz15ACM}, so $\Xi(E)$ is indeed a $\m B_0$-module. 

On the other hand, recall that for an \'{e}tale double cover $p:\td{\Delta}\to \Delta$, there is an associated 2-torsion line bundle $\pi:\xi\to \Delta$ such that $\td{\Delta}$ is recovered as the cyclic cover of $\xi$ and the section $1\in \xi^{\otimes 2}\cong \O_\Delta$ i.e. $\td{\Delta}$ is embedded in  $\textbf{Tot}(\xi)$  as the zero locus of $  t^{\otimes 2}-\pi^* 1$ where $t$ is the tautological section of $\pi^*\xi.$ Recall that a $\xi$-twisted $SL_2$-\emph{Higgs bundle} on a curve $\sm$ is a pair $(V,\phi)$ consisting of a rank 2 vector bundle $V$ with a fixed determinant line bundle $L$ and $\phi\in H^0(\sm, \m End_0(V)\otimes \xi)$. Since we will only deal with this case, We simply call it a twisted Higgs bundle. The spectral correspondence \cite{beauville1989spectral} says that pushing forward a line bundle $N$ on $\td{\Delta}$ gives a twisted Higgs bundle $(p_*N, p_*t) $ on $\Delta$. In fact, $\Prym(\td{\Delta},\Delta)$ parametrizes all twisted Higgs bundles on $\Delta$ with the spectral curve defined by $t^{\otimes 2}- \pi^* 1$. Since the Hitchin base $H^0(\Delta, \xi^{\otimes 2})= H^0(\Delta, \O_\Delta) = \C$, all smooth spectral curves (defined away from $0\in \C$) are isomorphic to each other.

Combining the functor $\Xi$ which induces a birational map $\f M_Y\dashrightarrow \f M_{2,-4}$, the birational map $\Phi: \f M_{2,-4} \dashrightarrow W^1_2$, the Abel-Jacobi map $\td{\a}: \td{\Delta}^{(10)} \to J^{10}\td{\Delta}$ and the spectral correspondence, we obtain an explicit correspondence between instanton bundles on $Y$ and $\xi$-twisted Higgs bundles on $\Delta$:
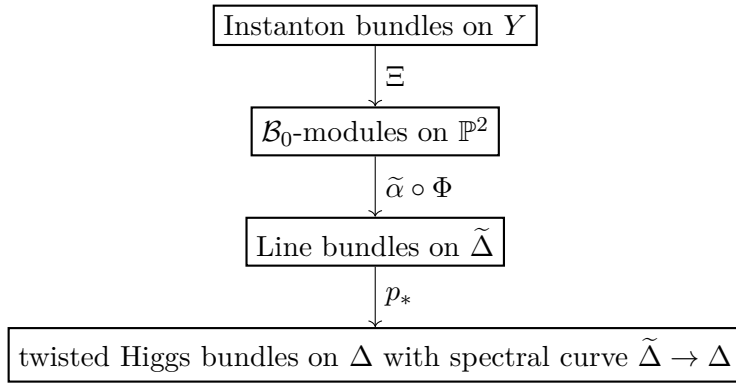
    \begin{center}
     \begin{tikzpicture}[
     squarednode1/.style={rectangle, draw=black!700,  thick, minimum size=5mm},
     ]
     \node[squarednode1](b1){Instanton bundles on $Y$}; 
     \node[squarednode1](b2)[below=0.8cm of b1] {$\m B_0$-modules on $\P^2$ };
     \node[squarednode1](b3)[below=0.8cm of b2] {Line bundles on $\td{\Delta}$};
     \node[squarednode1](b4)[below=0.8cm of b3] {twisted Higgs bundles on $\Delta$ with spectral curve $\td{\Delta}\to \Delta$};
     \draw[->] (b1) -- node[right]{$\Xi$} (b2);
     \draw[->] (b2) -- node[right]{$\td{\a}\circ \Phi$} (b3);
     \draw[->] (b3) -- node[right]{$p_*$} (b4);
     \end{tikzpicture}

\end{center}
The correspondences of different objects here hold as birational maps between the corresponding moduli spaces.

\printbibliography

@article {Simpson94moduli1,
    AUTHOR = {Simpson, Carlos T.},
     TITLE = {Moduli of representations of the fundamental group of a smooth
              projective variety. {I}},
   JOURNAL = {Inst. Hautes \'{E}tudes Sci. Publ. Math.},
  FJOURNAL = {Institut des Hautes \'{E}tudes Scientifiques. Publications
              Math\'{e}matiques},
    NUMBER = {79},
      YEAR = {1994},
     PAGES = {47--129},
      ISSN = {0073-8301},
   MRCLASS = {14D20 (14D22 14D25 14F05)},
  MRNUMBER = {1307297},
MRREVIEWER = {Nitin Nitsure},
       URL = {http://www.numdam.org/item?id=PMIHES_1994__79__47_0},
}

@article {Lahoz15ACM,
    AUTHOR = {Lahoz, Mart{\'i} and Macr{\`i} , Emanuele and Stellari, Paolo},
     TITLE = {Arithmetically {C}ohen-{M}acaulay bundles on cubic threefolds},
   JOURNAL = {Algebr. Geom.},
  FJOURNAL = {Algebraic Geometry},
    VOLUME = {2},
      YEAR = {2015},
    NUMBER = {2},
     PAGES = {231--269},
      ISSN = {2313-1691},
   MRCLASS = {14F05 (14E05 14J60)},
  MRNUMBER = {3350158},
MRREVIEWER = {Cristian V. Anghel},
       DOI = {10.14231/AG-2015-011},
       URL = {https://doi.org/10.14231/AG-2015-011},
}

@article{Iliev2000,
author = {Iliev, A. and Markushevich, D.},
journal = {Documenta Mathematica},
language = {eng},
pages = {23-47},
publisher = {Universiät Bielefeld, Fakultät für Mathematik},
title = {The Abel-Jacobi map for cubic threefold and periods of Fano threefolds of degree 14.},
url = {http://eudml.org/doc/48423},
volume = {5},
year = {2000},
}

@article{markushevich1998,
author = {Markushevich, Dimitri and Tikhomirov, Alexander},
year = {1998},
month = {10},
pages = {},
title = {The Abel-Jacobi map of a moduli component of vector bundles on the cubic threefold},
volume = {10},
journal = {Journal of Algebraic Geometry}
}

@article{kuznetsov2008derived,
  title={Derived categories of quadric fibrations and intersections of quadrics},
  author={Kuznetsov, Alexander},
  journal={Advances in Mathematics},
  volume={218},
  number={5},
  pages={1340--1369},
  year={2008},
  publisher={Elsevier}
}

@article{clemens1972intermediate,
  title={The intermediate Jacobian of the cubic threefold},
  author={Clemens, C Herbert and Griffiths, Phillip A},
  journal={Annals of Mathematics},
  pages={281--356},
  year={1972},
  publisher={JSTOR}
}

@article{beauville1982sous,
  title={Sous-vari{\'e}t{\'e}s sp{\'e}ciales des vari{\'e}t{\'e}s de Prym},
  author={Beauville, Arnaud},
  journal={Compositio Mathematica},
  volume={45},
  number={3},
  pages={357--383},
  year={1982}
}

@ARTICLE{Druel, 
author={S. {Druel}},  journal={International Mathematics Research Notices},   title={Espace des modules des faisceaux de rang 2 semi-stables de classes de Chern c 1 = 0, c 2 = 2 et c 3 = 0 sur la cubique de $\mathbb{P}^4$},   year={2000},  volume={2000},  number={19},  pages={985-1004},}

@incollection {beauville02cubic,
    AUTHOR = {Beauville, Arnaud},
     TITLE = {Vector bundles on the cubic threefold},
 BOOKTITLE = {Symposium in {H}onor of {C}. {H}. {C}lemens ({S}alt {L}ake
              {C}ity, {UT}, 2000)},
    SERIES = {Contemp. Math.},
    VOLUME = {312},
     PAGES = {71--86},
 PUBLISHER = {Amer. Math. Soc., Providence, RI},
      YEAR = {2002},
   MRCLASS = {14J60 (14J28 14J30)},
  MRNUMBER = {1941574},
MRREVIEWER = {Carlo Giovanni Madonna},
       DOI = {10.1090/conm/312/04987},
       URL = {https://doi.org/10.1090/conm/312/04987},
}

@book {assem_simson_skowronski06elements,
    AUTHOR = {Assem, Ibrahim and Simson, Daniel and Skowro\'{n}ski, Andrzej},
     TITLE = {Elements of the representation theory of associative algebras.
              {V}ol. 1},
    SERIES = {London Mathematical Society Student Texts},
    VOLUME = {65},
      NOTE = {Techniques of representation theory},
 PUBLISHER = {Cambridge University Press, Cambridge},
      YEAR = {2006},
     PAGES = {x+458},
   MRCLASS = {16G10 (16-02)},
  MRNUMBER = {2197389},
MRREVIEWER = {Peter W. Donovan},
       DOI = {10.1017/CBO9780511614309},
       URL = {https://doi.org/10.1017/CBO9780511614309},
}

@article {auel14quadrics,
    AUTHOR = {Auel, Asher and Bernardara, Marcello and Bolognesi, Michele},
     TITLE = {Fibrations in complete intersections of quadrics, {C}lifford
              algebras, derived categories, and rationality problems},
   JOURNAL = {J. Math. Pures Appl. (9)},
  FJOURNAL = {Journal de Math\'{e}matiques Pures et Appliqu\'{e}es. Neuvi\`eme S\'{e}rie},
    VOLUME = {102},
      YEAR = {2014},
    NUMBER = {1},
     PAGES = {249--291},
      ISSN = {0021-7824},
   MRCLASS = {14F05 (11E08 11E88 14E08 14J26)},
  MRNUMBER = {3212256},
MRREVIEWER = {Mihnea Popa},
       DOI = {10.1016/j.matpur.2013.11.009},
       URL = {https://doi.org/10.1016/j.matpur.2013.11.009},
}

@article {borne07rootstacks,
    AUTHOR = {Borne, Niels},
     TITLE = {Fibr\'{e}s paraboliques et champ des racines},
   JOURNAL = {Int. Math. Res. Not. IMRN},
  FJOURNAL = {International Mathematics Research Notices. IMRN},
      YEAR = {2007},
    NUMBER = {16},
     PAGES = {Art. ID rnm049, 38},
      ISSN = {1073-7928},
   MRCLASS = {14F05},
  MRNUMBER = {2353089},
MRREVIEWER = {Nicolas Perrin},
       DOI = {10.1093/imrn/rnm049},
       URL = {https://doi.org/10.1093/imrn/rnm049},
}

@article {beauville1989spectral,
    AUTHOR = {Beauville, Arnaud and Narasimhan, M. S. and Ramanan, S.},
     TITLE = {Spectral curves and the generalised theta divisor},
   JOURNAL = {J. Reine Angew. Math.},
  FJOURNAL = {Journal f\"{u}r die Reine und Angewandte Mathematik. [Crelle's
              Journal]},
    VOLUME = {398},
      YEAR = {1989},
     PAGES = {169--179},
      ISSN = {0075-4102},
   MRCLASS = {14H60 (14D20 14H42)},
  MRNUMBER = {998478},
MRREVIEWER = {P. E. Newstead},
       DOI = {10.1515/crll.1989.398.169},
       URL = {https://doi.org/10.1515/crll.1989.398.169},
}

@article {Cadman07rootstack,
    AUTHOR = {Cadman, Charles},
     TITLE = {Using stacks to impose tangency conditions on curves},
   JOURNAL = {Amer. J. Math.},
  FJOURNAL = {American Journal of Mathematics},
    VOLUME = {129},
      YEAR = {2007},
    NUMBER = {2},
     PAGES = {405--427},
      ISSN = {0002-9327},
   MRCLASS = {14D20 (14A20 14N35)},
  MRNUMBER = {2306040},
MRREVIEWER = {Michael A. Rose},
       DOI = {10.1353/ajm.2007.0007},
       URL = {https://doi.org/10.1353/ajm.2007.0007},
}

@book {Welters81AJ_Isogenies,
    AUTHOR = {Welters, G. E.},
     TITLE = {Abel-{J}acobi isogenies for certain types of {F}ano
              threefolds},
    SERIES = {Mathematical Centre Tracts},
    VOLUME = {141},
 PUBLISHER = {Mathematisch Centrum, Amsterdam},
      YEAR = {1981},
     PAGES = {i+139},
      ISBN = {90-6196-227-7},
   MRCLASS = {14K30 (14J30)},
  MRNUMBER = {633157},
MRREVIEWER = {A. S. Tikhomirov},
}

@article {mumford71theta,
    AUTHOR = {Mumford, David},
     TITLE = {Theta characteristics of an algebraic curve},
   JOURNAL = {Ann. Sci. \'{E}cole Norm. Sup. (4)},
  FJOURNAL = {Annales Scientifiques de l'\'{E}cole Normale Sup\'{e}rieure. Quatri\`eme
              S\'{e}rie},
    VOLUME = {4},
      YEAR = {1971},
     PAGES = {181--192},
      ISSN = {0012-9593},
   MRCLASS = {14H99},
  MRNUMBER = {292836},
MRREVIEWER = {P. E. Newstead},
       URL = {http://www.numdam.org/item?id=ASENS_1971_4_4_2_181_0},
}

@incollection {mumford74prym,
    AUTHOR = {Mumford, David},
     TITLE = {Prym varieties. {I}},
 BOOKTITLE = {Contributions to analysis (a collection of papers dedicated to
              {L}ipman {B}ers)},
     PAGES = {325--350},
      YEAR = {1974},
   MRCLASS = {14H40 (14K25 32G20)},
  MRNUMBER = {0379510},
MRREVIEWER = {H. H. Martens},
}

@article {bernardara12cubic,
    AUTHOR = {Bernardara, Marcello and Macr{\`i} , Emanuele and Mehrotra,
              Sukhendu and Stellari, Paolo},
     TITLE = {A categorical invariant for cubic threefolds},
   JOURNAL = {Adv. Math.},
  FJOURNAL = {Advances in Mathematics},
    VOLUME = {229},
      YEAR = {2012},
    NUMBER = {2},
     PAGES = {770--803},
      ISSN = {0001-8708},
   MRCLASS = {14F05 (14C34 14J30)},
  MRNUMBER = {2855078},
       DOI = {10.1016/j.aim.2011.10.007},
       URL = {https://doi.org/10.1016/j.aim.2011.10.007},
}

@misc{bayer2017stability,
      title={Stability conditions on Kuznetsov components}, 
      author={Arend Bayer and Martí Lahoz and Emanuele Macr{\`i} and Paolo Stellari},
      year={2017},
      eprint={1703.10839},
      archivePrefix={arXiv},
      primaryClass={math.AG}
}

@article {bayer2021stability,
    AUTHOR = {Bayer, Arend and Lahoz, Marti and Macr{\`i} , Emanuele and Nuer,
              Howard and Perry, Alexander and Stellari, Paolo},
     TITLE = {Stability conditions in families},
   JOURNAL = {Publ. Math. Inst. Hautes \'{E}tudes Sci.},
  FJOURNAL = {Publications Math\'{e}matiques. Institut de Hautes \'{E}tudes
              Scientifiques},
    VOLUME = {133},
      YEAR = {2021},
     PAGES = {157--325},
      ISSN = {0073-8301},
   MRCLASS = {14F08 (14J42)},
  MRNUMBER = {4292740},
       DOI = {10.1007/s10240-021-00124-6},
       URL = {https://doi.org/10.1007/s10240-021-00124-6},
}

@article {kuznetsov2012instanton,
    AUTHOR = {Kuznetsov, Alexander},
     TITLE = {Instanton bundles on {F}ano threefolds},
   JOURNAL = {Cent. Eur. J. Math.},
  FJOURNAL = {Central European Journal of Mathematics},
    VOLUME = {10},
      YEAR = {2012},
    NUMBER = {4},
     PAGES = {1198--1231},
      ISSN = {1895-1074},
   MRCLASS = {14J60 (14F05 14J30 14J45)},
  MRNUMBER = {2925598},
MRREVIEWER = {Francesco Malaspina},
       DOI = {10.2478/s11533-012-0055-1},
       URL = {https://doi.org/10.2478/s11533-012-0055-1},
}

@incollection {kuznetsov2010cubic,
    AUTHOR = {Kuznetsov, Alexander},
     TITLE = {Derived categories of cubic fourfolds},
 BOOKTITLE = {Cohomological and geometric approaches to rationality
              problems},
    SERIES = {Progr. Math.},
    VOLUME = {282},
     PAGES = {219--243},
 PUBLISHER = {Birkh\"{a}user Boston, Boston, MA},
      YEAR = {2010},
   MRCLASS = {14F05 (14E05)},
  MRNUMBER = {2605171},
MRREVIEWER = {Paolo Stellari},
       DOI = {10.1007/978-0-8176-4934-0\_9},
       URL = {https://doi.org/10.1007/978-0-8176-4934-0_9},
}

@article {lepotier1993,
    AUTHOR = {Le Potier, J.},
     TITLE = {Faisceaux semi-stables de dimension {$1$} sur le plan
              projectif},
   JOURNAL = {Rev. Roumaine Math. Pures Appl.},
  FJOURNAL = {Revue Roumaine de Math\'{e}matiques Pures et Appliqu\'{e}es. Romanian
              Journal of Pure and Applied Mathematics},
    VOLUME = {38},
      YEAR = {1993},
    NUMBER = {7-8},
     PAGES = {635--678},
      ISSN = {0035-3965},
   MRCLASS = {14D20 (14F05 14J60)},
  MRNUMBER = {1263210},
MRREVIEWER = {Jean D'Almeida},
}

\end{document}